\documentclass[12pt]{article}

\usepackage{amsfonts,amssymb,amsmath,exscale,relsize,enumitem,amsthm,mathtools, physics}
\usepackage{graphicx}
\usepackage{float}
\usepackage[thinlines]{easytable}
\usepackage{csquotes}
\usepackage{setspace}

\usepackage[authoryear,round]{natbib}

\usepackage[
colorlinks=true,
citecolor=blue,
linkcolor=blue,
urlcolor=blue,
backref=page
]{hyperref}

\usepackage{varioref}
\usepackage{subcaption}
\usepackage{xcolor}
\usepackage{setspace}
\usepackage{appendix}
\usepackage[margin=1.35in]{geometry}
\usepackage{setspace} 

\DeclareMathOperator*{\argmax}{argmax}

\newtheorem{remark}{Remark}

\def\prime{\mathrm{\scriptscriptstyle{T}}}

\mathchardef\mhyphen="2D

\newtheorem{thm}{Theorem}
\newtheorem{lem}{Lemma}
\newtheorem{cor}{Corollary}

\title{\textbf{Admissibility of Adaptive Monotone Step-Down Multiple Testing Procedures Under Arbitrary Covariance Dependence}}

\author{Prasenjit Ghosh\thanks{\noindent{Department of Statistics,
	Texas A\&M University, College Station, TX 77843, USA; and Applied
	Statistics Unit, Indian Statistical Institute, Kolkata 700108, India.
	Email: \texttt{prasenjit.tamu@gmail.com}.}}
	\hspace{1em} 
	Arijit Chakrabarti\thanks{\noindent{Applied Statistics Unit, Indian Statistical Institute, Kolkata - 700108, India. Email: \texttt{arc@isical.ac.in}}}%
}

\date{} 

\newcommand{\keywords}[1]{%
	\par\addvspace\baselineskip
	\noindent\textbf{Keywords: }#1
}

\begin{document}
	
	\maketitle
	
\begin{abstract}
	We study finite-sample admissibility for simultaneous testing of multivariate
	normal means under a vector loss whose components are the usual individual
	$0$--$1$ losses, with an arbitrary known positive-definite covariance
	matrix. For the two-sided problem, we establish admissibility for a broad
	class of procedures that rank active hypotheses using locally adaptive,
	strictly increasing transformations of suitably standardized conditional
	residuals. The result also covers residual-based sequential procedures with
	positive stage-specific cutoffs common to all active hypotheses at each
	stage, without requiring the cutoffs to be nonincreasing. The proof derives
	the interval structure of the acceptance sections from two properties:
	the previously rejected indices remain invariant while the hypothesis
	under study is active, and the stopping rule ensures that a perturbation
	leading to its rejection cannot defer that rejection beyond the stage at
	which the procedure at the unperturbed observation stops while accepting
	it. We also establish the corresponding one-sided result for locally
	adaptive, strictly increasing transformations of signed residuals, under
	both point-null and composite-null formulations and their respective vector
	losses. Together, these results extend vector-loss admissibility beyond the original
	Maximum Residual Down (MRD) procedure and identify a broader monotone
	ordering structure induced by covariance-adjusted residuals. We also show that the two-sided and one-sided classes are almost-everywhere limits of finite-support Bayes rules under their corresponding additive total losses; this representation, however, does not assert admissibility under those losses.
\end{abstract}


\vspace{0.4cm}

\section{Introduction}\label{SECTION_INTRODUCTION}
Multiple hypothesis testing has emerged as one of the central themes in modern
statistical inference, particularly in the analysis of high-dimensional datasets
arising in genomics, bioinformatics, brain imaging, astronomy, economics,
finance, medicine, and related scientific disciplines. In many contemporary
applications, a very large number of hypotheses must be tested simultaneously
in order to identify a relatively small number of significant signals hidden
among a vast collection of null effects. Such large-scale inferential problems
require statistical procedures that effectively distinguish true signals from
noise while maintaining adequate control over erroneous decisions at a global
level.

The need for simultaneous inference naturally introduces multiplicity issues,
since procedures designed for individual testing problems may accumulate errors
rapidly when applied repeatedly across a large collection of hypotheses.
Consequently, a vast literature has developed on multiple testing procedures
controlling global error criteria such as the family-wise error rate (FWER)
and the false discovery rate (FDR); see, for example, \citet{BH1995},
\citet{BL1999}, \citet{BY2001}, \citet{STO_2002}, \citet{STS_2004},
\citet{BKY2006}, \citet{SKS2006}, \citet{SKS2008},  \citet{BLROQ2009}, \citet{SUN_CAI_2009}, \cite{FHG2012} and the references
therein. While classical FWER-controlling procedures often become excessively
conservative in large-scale settings, FDR-based methodologies typically provide
a more flexible trade-off between false discoveries and false non-discoveries,
and have therefore become central to modern large-scale multiple testing.

A substantial part of the early theoretical work on multiple testing assumed that the test
statistics were independent. In many scientific applications, however,
dependence among test statistics arises from spatial, temporal, biological, or network structure.
Ignoring dependence can alter the operating characteristics and stability
of multiple-testing procedures. Studies have examined its effects on
false-discovery counts and the stability of selected sets; see, for
instance, \citet{GGQY2007}, \citet{KY2007}, \citet{QBKY2005},
\citet{QKY2005}, and \citet{QXGY2007}. \citet{Efron_2007} also emphasized
that ignoring dependence can lead to misleading inferences in highly
correlated settings. Dependence can, however, provide structure that
improves inferential efficiency; see \citet{HJ2010}, \citet{BH2000}, and
\citet{GRW2006}.

These considerations have generated a broad literature on multiple testing
when test statistics are dependent. The approaches differ in how they
address dependence: some exploit or estimate its structure, whereas others
establish error-control guarantees under specified dependence conditions
without explicitly modeling it. Examples include \citet{BY2001},
\citet{ROM_SHK_WOLF}, \citet{SKS2008}, \citet{LEEK_STO_2008},
\citet{BLROQ2009}, \citet{FRI_KLO_CAU_2009}, and \citet{SUN_CAI_2009}.
Recent work continues to study FDR control when dependence is present. For
known dependence, \citet{FithianLei2022ConditionalCalibration} developed
conditional calibration, while \citet{SarkarZhang2025ShiftedBH} and
\citet{GhoshSarkar2025DependenceAware} proposed shifted and generalized
shifted Benjamini--Hochberg (BH) procedures for correlated Gaussian mean testing. The e-BH
procedure of \citet{WangRamdas2022Ebh} controls FDR under arbitrary
dependence when valid $e$-values are available; related work develops
conditional boosting and compound $e$-values
\citep{LeeRen2024BoostingEBH, IgnatiadisWangRamdas2026}. For a broader
account of $e$-value methods, see
\citet{WangRamdas2022Ebh, RamdasWang2025EvaluesMonograph}. More recently,
\citet{Dobriban2026BHFailure} showed that BH can exceed its nominal FDR
level in some correlated two-sided Gaussian testing problems. These
contributions study the global error-rate properties of particular
procedures under specified dependence conditions. Such guarantees are
important, but they do not by themselves compare procedures according to
their risks across the parameter space. The complementary
decision-theoretic question is whether another procedure can have no
larger risk throughout the parameter space and strictly smaller risk
somewhere, under a specified loss.

This question is especially relevant under dependence. Several classical
$p$-value-based stepwise procedures, including the BH procedure, have been shown to be inadmissible under vector-valued
losses in correlated models
\citep{COHEN_KOL_SACK_2007, COHEN_SACK_2005_B, COHEN_SACK_2007,
	COHEN_SACK_2008}. Here, inadmissibility means that another procedure
has no greater risk in any component at every parameter value, and
strictly smaller risk in at least one component at some parameter
value. For the componentwise $0$--$1$ vector loss,
such domination also implies domination under the corresponding
additive total loss, because total risk is the sum of the component
risks. Thus, inadmissibility under this vector loss entails
inadmissibility under the matching total loss, though the converse need
not hold. These results show that error-rate control and risk admissibility address
different questions: an error-rate guarantee alone does not rule out a
uniform risk improvement. We therefore study admissibility directly,
under the componentwise risks specified below.

Motivated by this, we study in this paper finite-sample admissibility in a multivariate Gaussian mean model. We observe
a random vector $\boldsymbol{X}\sim
N_n(\boldsymbol{\theta},\boldsymbol{\Sigma})$,
where $\boldsymbol{\theta}$ is an unknown fixed mean vector and
$\boldsymbol{\Sigma}$ is known and positive definite. Our primary problem
is simultaneous testing of the coordinatewise point nulls
$H_{0i}:\theta_i=0$ against the two-sided alternatives
$H_{Ai}:\theta_i\neq 0$, for $i=1,\ldots,n$. We impose no structural
restriction on $\boldsymbol{\Sigma}$ beyond positive definiteness. We
evaluate procedures under a vector of componentwise $0$--$1$ losses.
Thus, our objective is risk admissibility, rather than FWER or FDR control.

This formulation connects the paper to the decision-theoretic literature on
multiple testing, which studies admissibility under scalar and vector losses
\citep{COHEN_SACK_2005_A, COHEN_SACK_2005_B, DSB2003, FS2002}. A related,
but distinct, literature studies procedures that are undominated within
classes constrained to preserve an error-control guarantee. It includes
uniform improvements of BH and adaptive BH rules
\citep{SolariGoeman2017, Goeman2026ClosedBH,
	IgnatiadisWangRamdas2026} and complete-class results for e-value-based FDR
procedures \citep{SunWang2026EvaluesCompleteClass}. Other work studies
admissibility for false-discovery-proportion control
\citep{GoemanHemerikSolari2021}, the merging of $p$-values under arbitrary
dependence \citep{VovkWangWang2022}, and the merging of $e$-values
\citep{Wang2025, Clerico2026}. These works use validity-preserving
dominance criteria. They provide related perspectives, but do not establish
risk admissibility for the residual step-down procedures considered here.

A major development in the risk-based study of dependent multiple testing
was the Maximum Residual Down (MRD) procedure of
\citet{COHEN_SACK_XU_2009}. MRD incorporates dependence through adaptively
formed residual statistics derived from conditional normal distributions
and was shown to be admissible under a vector-valued loss. The paper also
treats point-null and composite-null one-sided formulations, using signed
rather than absolute residuals. Earlier work treats one-sided
multiple-comparison procedures under vector risk in more specialized
models; see, for example, \citet{COHEN_SACK_2005_A}. The one-sided result
below extends this line of work to locally adaptive signed-residual scores
under arbitrary known positive-definite covariance.

%

However, the admissibility proof for the MRD procedure is closely tied to the
specific structure of the residual statistics employed there, and the underlying
structural principle responsible for admissibility remains somewhat implicit.
In particular, because the residual statistics are constructed through adaptive
stagewise index selections depending on the observed data vector itself, the
associated acceptance regions possess highly nontrivial data-dependent geometry.
As a consequence, it is not immediately clear whether admissibility is a special
feature of the precise MRD construction itself, or whether it reflects a deeper
structural phenomenon associated with covariance-adapted residual geometry
and more general locally adaptive residual-based step-down mechanisms.

The residual statistics underlying the MRD procedure naturally induce a stagewise ordering of the active hypotheses according to the strength of the
corresponding residual evidence against the null hypotheses. Larger residual
magnitudes represent stronger conditional evidence against the
corresponding null hypotheses, making monotone residual scoring a natural
statistical requirement: stronger residual evidence should not receive a
weaker score. This motivates studying procedures that transform the MRD
residuals through strictly increasing functions, while allowing those
functions to vary across stages, coordinates, and rejection histories.
Such variation can change the relative rankings of the active hypotheses
and the resulting rejection path. The resulting class therefore lets us
ask whether admissibility depends on MRD's particular scoring formula or
persists under a broader, locally adaptive monotone treatment of
covariance-adjusted residual evidence.

The class considered here includes procedures with distinct forms of
residual scoring. The Bayesian Step-Down (BSD) procedure proposed by
\citet{GC_BSD2015} provides a nontrivial example: under the prior
specifications considered there, its stagewise posterior scores can be
expressed as locally adaptive, strictly increasing transformations of
the absolute MRD residuals. The class also includes the MRD--GBS
procedure proposed by \citet{GhoshChakrabartiMRDGBS2026}, which uses
the identity score with the Gavrilov--Benjamini--Sarkar stagewise
critical values \citep{GBS2009}.

Recent companion work studies further properties of these procedures.
\citet{GhoshChakrabartiBSD2026} develops the Bayesian model-pursuit
formulation of BSD and studies near-oracle sparse signal discovery
under dependence. Under one-factor dependence, the ABOS analysis in
\citet{GhoshChakrabartiOneFactorABOS2026} establishes ABOS for
MRD--GBS under its stated sparse-regime assumptions and for the
oracle-parameterized BSD procedure when the true model, covariance,
and loss parameters are supplied. These are distinct asymptotic
risk results; they complement but do not follow from the finite-sample
admissibility results established here.

The present paper extends the MRD admissibility theory to locally adaptive
scores. In the two-sided problem, each stage may use a strictly increasing
transformation of the absolute standardized MRD residual, with the
transformation varying across stages, coordinates, and rejection histories.
For arbitrary known positive-definite covariance, the resulting procedures
are admissible under componentwise $0$--$1$ vector loss. The proof also shows
that the stage cutoffs need only be positive and common to the active
hypotheses at each stage; they need not be nonincreasing. Rules with an
increasing cutoff sequence therefore fall within the result, although they
need not be step-down procedures in the conventional sense.

This extension is not a formal consequence of MRD's admissibility. A change
in the selected hypothesis can alter both the next conditioning set and the
score transformation used at the next stage, so the effect can propagate
through the rejection path. The proof establishes the acceptance-section
geometry directly. The residual-shift identity keeps competing residuals
fixed while the target hypothesis is active; path invariance and the
stopping-stage argument then show that its acceptance section is an interval.
The Matthes--Truax criterion converts this one-dimensional geometry into
admissibility. Thus, the result covers procedures whose adaptive rejection
paths and stagewise thresholds can differ from MRD's.

The one-sided extension uses signed residuals because only positive residual
evidence supports rejection. We treat both $H_{0i}:\theta_i=0$ versus
$H_{Ai}:\theta_i>0$ on $[0,\infty)^n$ and
$H_{0i}:\theta_i\leq0$ versus $H_{Ai}:\theta_i>0$ on $\mathbb{R}^n$, with
their respective componentwise vector losses. For each formulation, we extend
the one-sided MRD result of \citet{COHEN_SACK_XU_2009} to locally adaptive,
strictly increasing transformations of signed residuals. The signed residual
increases along the relevant Matthes--Truax direction while competing active
residuals remain fixed, giving a lower half-line acceptance section. The
one-sided result, like the two-sided result, requires only a positive
cutoff common to the active hypotheses at each stage; nonincreasing cutoffs
specify the conventional step-down subclass. Neither result asserts
admissibility under additive total loss or FDR control.



The vector-loss results leave open a distinct and more fundamental decision-theoretic question:
admissibility under additive total misclassification loss. Under vector
loss, a dominator must improve each component risk separately, so the
problem can be reduced to the admissibility of the induced individual
tests. Under total loss, however, a risk reduction for some hypotheses may
be offset by a risk increase for others. The gap persists even under independence, although there is an important
low-dimensional exception. For independent normal observations with
$n=2$, \citet[Theorem~5.3]{COHEN_SACK_2005_A} establish total-loss
admissibility of a step-down procedure for the composite one-sided
problem $H_i:\mu_i\leq 0$ versus $K_i:\mu_i>0$, with equal unit costs
for the two types of error. This result does not cover the two-sided
problem studied here. In the same paper, the step-down rule is also
shown to be a limit of Bayes procedures under independence, but that
representation alone does not establish admissibility. To our knowledge, these results do not establish total-loss admissibility
for the two-sided problem; in particular, admissibility of GBS-type
step-down rules for $n\geq 3$ under independence remains open.

%
%

Arbitrary dependence adds a further obstacle. Along a natural-coordinate
line, competing residuals remain fixed only while the hypothesis under
study is active. Once it is rejected, the conditioning set changes, and
so may the residuals and decisions for the remaining hypotheses. The
interval structure that supports the vector-loss argument therefore
does not control the aggregate risk difference. A total-loss
admissibility proof must handle cross-hypothesis risk trade-offs and
changes in the adaptive rejection path uniformly over the parameter
space.

The unresolved total-loss question motivates a complementary construction:
can these adaptive procedures be approximated by Bayes rules under the same
aggregate loss? This question is related to complete-class theory, where,
under appropriate regularity and risk-topology conditions, suitable
closures of Bayes rules can form essentially complete classes
\citep{Wald1947EssentiallyCompleteClass,Berger1985SDT}. The topology and
assumptions matter. Our result gives a procedure-specific finite-support
Bayes representation and, as shown below, pointwise convergence of the
associated risk functions; it does not invoke a complete-class theorem or
infer admissibility from Bayes closure.

Bayes-limit representations of step-down procedures are not without
precedent. For the classical one-sided multiple-endpoint problem under
independence, \citet{COHEN_SACK_2005_A} show that a step-down rule is a
limit of Bayes rules under additive loss. To the best of our knowledge,
the present result is the first of its kind to establish such a representation for
covariance-adaptive residual step-down procedures under arbitrary known
positive-definite covariance. Every procedure in our two-sided class and
our signed one-sided class is an almost-everywhere limit, in natural
coordinates, of Bayes rules generated by finite-support priors under the
corresponding total loss. In the one-sided cases, those priors can be
supported entirely on the nonnegative orthant.

The construction turns the locality of the conditional residuals into a
representation of the full adaptive rejection path by prior atoms. This
is the bridge to the paper’s vector-loss results: the same residual
geometry underlies both arguments, while the Bayes-limit theorem places
the procedures in a Bayes-rule closure under additive total loss. That
connection is substantive, but its reach is precise: Bayes-limit
representation does not establish admissibility under total loss, which
remains open.

%
%

The remainder of the paper is organized as follows.
Section~\ref{SECTION_PROBLEM_FORMULATION} formulates the two-sided multiple
testing problem, introduces its vector-loss framework, and describes the
residual-based step-down procedures.
Section~\ref{SECTION_TWO_SIDED_ADMISSIBILITY} develops the two-sided
admissibility theory, and Section~\ref{SECTION_ONE_SIDED_ADMISSIBILITY}
presents the one-sided formulations and their admissibility result.
Section~\ref{SECTION_BAYES_LIMITS} establishes the Bayes-limit
representations under additive total loss and explains the residual
locality, path encoding, and finite-prior construction behind them.
Section~\ref{SECTION_DISCUSSION}
discusses the implications and limitations of these results and directions
for future research. Proofs are collected in Appendix~\ref{SECTION_APPENDIX}.

\section{Problem Formulation and Residual-Based Step-Down Testing Procedures}
\label{SECTION_PROBLEM_FORMULATION}

In this paper, we consider the problem of simultaneous testing for means of a set of jointly normal variables. Towards that, let us assume that we observe a random vector $\boldsymbol{X}=(X_1,\dots,X_n)$ (obtained through some suitable transformation, if necessary) such that $\boldsymbol{X} \sim N_{n}(\boldsymbol{\theta},\boldsymbol{\Sigma})$ where  $\boldsymbol{\theta}=(\theta_1,\dots,\theta_n)$ is the vector of unknown means and $\boldsymbol{\Sigma}=((\sigma_{ij}))$ is an $n \times n$ known positive definite matrix with an arbitrary covariance structure. We are interested in testing simultaneously 
\begin{equation}\label{TESTING_PROBLEM}
	H_{0i}:\theta_{i}=0 \mbox{ against } H_{Ai}:\theta_{i}\neq 0, \mbox{ for } i=1,\dots,n.
\end{equation}

Because positive coordinatewise rescaling preserves the null hypotheses
and the componentwise losses, one may equivalently standardize the
marginal variances to one. We retain the general covariance notation
throughout.

We now introduce the Maximum Residual Down (MRD) procedure proposed by \citet{COHEN_SACK_XU_2009}, which serves as the canonical residual-based
step-down procedure under arbitrary covariance dependence. For that, we adopt here similar convention of notations used in \citet{COHEN_SACK_XU_2009}. 

Let $\boldsymbol{X}^{(i_{1},\dots,i_{t})}$ be an $(n-t)\times1$ vector consisting of those components of $\boldsymbol{X}=(X_1,\dots,X_n)$ with $(X_{i_{1}},\dots,X_{i_{t}})$ left out. Suppose $\boldsymbol{\Sigma}_{(i_{1},\dots,i_{t})} $ is the $(n-t)\times(n-t)$ sub-matrix obtained after eliminating the $i_{1},\dots,i_{t}$-th rows and the corresponding columns of $\boldsymbol{\Sigma}$. Let $\boldsymbol{\sigma}_{(j)}^{(i_{1},\dots,i_{t})} $ be the $(n-t-1)\times1$ vector obtained by eliminating the $i_{1},\dots,i_{t}$-th and $j$-th elements of the $j$-th column vector of $\boldsymbol{\Sigma}$. Define
\begin{eqnarray}
	\sigma_{j\cdot(i_{1},\dots,i_{t})} &=& \sigma_{jj} - {\boldsymbol{\sigma}_{(j)}^{(i_{1},\dots,i_{t})}}^{T}{\boldsymbol{\Sigma}^{-1}_{(i_1,\dots,i_t,j)}}{\boldsymbol{\sigma}_{(j)}^{(i_{1},\dots,i_{t})}}. \nonumber
\end{eqnarray}

The MRD procedure is based on a collection of adaptively formed residual
statistics defined as
\begin{eqnarray}\label{MRD_STATISTICS}
	U_{tj}^{(i_1,\dots,i_{t-1})}(\boldsymbol{X}) &=& \frac{X_j - {\boldsymbol{\sigma}_{(j)}^{(i_1,\dots,i_{t-1})}}^{T}{\boldsymbol{\Sigma}^{-1}_{(i_1,\dots,i_{t-1},j)}}{\boldsymbol{X}^{(i_1,\dots,i_{t-1},j)}}}{\sigma^{\frac{1}{2}}_{j\cdot(i_1,\dots,i_{t-1})}},\nonumber
\end{eqnarray}
for $t,j=1,\dots,n$, $1\leqslant i_{1}\neq\dots\neq i_{t-1}\leqslant n$ and $i_{l}\neq j$ for all $l=1,\dots,t-1$.

For $1 \leqslant t \leqslant n$, we define the index $\widetilde{j}_t(\boldsymbol{X})$ as
\begin{align}\label{MRD_INDEX}
	\widetilde{j}_t(\boldsymbol{X})&=\argmax_{j \in \{1,\dots,n\}\setminus\{\widetilde{j}_1(\boldsymbol{X}),\dots,\widetilde{j}_{t-1}(\boldsymbol{X})\}} |U^{(\widetilde{j}_1(\boldsymbol{X}),\dots,\widetilde{j}_{t-1}(\boldsymbol{X}))}_{tj}(\boldsymbol{X})|.
\end{align}

Given a set of positive constants $C_1 \geqslant C_2 \geqslant \dots \geqslant C_n$, the MRD method works in a step-down manner as follows:
\begin{enumerate}
	\item At stage 1, consider the statistics $|U_{1j}(\boldsymbol{X})|$, where $j \in \{1,\dots,n\}$. If $|U_{1\widetilde{j}_1(\boldsymbol{X})}(\boldsymbol{X})| \leqslant C_1, $ stop and accept all $H_{0i}$'s. Otherwise reject $H_{0\widetilde{j}_1(\boldsymbol{X})}$ and continue to stage 2.
	
	\item At stage 2, consider the statistics $|U^{(\widetilde{j}_1(\boldsymbol{X}))}_{2j}(\boldsymbol{X})|$, where $j\in\{1,\dots,n \}\setminus\{\widetilde{j}_1(\boldsymbol{X})\}$. If $|U^{(\widetilde{j}_1(\boldsymbol{X}))}_{2\widetilde{j}_2(\boldsymbol{X})}(\boldsymbol{X})| \leqslant C_2, $ stop and accept all the remaining $H_{0i}$'s. Otherwise, reject $H_{0\widetilde{j}_2(\boldsymbol{X})}$ and continue to stage 3.
	
	\item In general, at stage $t$, consider the statistics $|U^{(\widetilde{j}_1(\boldsymbol{X}),\dots,\widetilde{j}_{t-1}(\boldsymbol{X}))}_{tj}(\boldsymbol{X})|$, where $j\in\{1,\dots,n\}\setminus\{\widetilde{j}_1(\boldsymbol{X}),\dots,\widetilde{j}_{t-1}(\boldsymbol{X})\}$. If $|U^{(\widetilde{j}_1(\boldsymbol{X}),\dots,\widetilde{j}_{t-1}(\boldsymbol{X}))}_{t\widetilde{j}_t(\boldsymbol{X})}(\boldsymbol{X})| \leqslant C_t$, stop and accept all the remaining $H_{0i}$'s. Otherwise, reject $H_{0\widetilde{j}_t(\boldsymbol{X})}$ and move to stage $(t+1)$.
	
	\item We continue in this fashion until an acceptance occurs or there are no more null hypotheses to be tested, in which case we must stop.
\end{enumerate}

The primary objective of the present paper is to identify the broader structural
principle underlying admissibility of residual-based step-down procedures under
dependence. Observe that the MRD procedure ranks the active hypotheses at each
stage according to the magnitudes of the corresponding residual statistics
\[
\left|
U_{tj}^{(i_1,\dots,i_{t-1})}(\boldsymbol{X})
\right|.
\]
This observation naturally leads to consideration of a more general class of
residual-based step-down procedures obtained through monotone transformations
of these absolute residuals.

The monotonicity assumption is natural from both statistical and decision-theoretic
perspectives. Intuitively, larger values of
\[
\left|
U_{tj}^{(i_1,\dots,i_{t-1})}(\boldsymbol{X})
\right|
\]
correspond to stronger residual evidence against the associated null hypothesis,
and hence any reasonable residual-based testing mechanism should respond
monotonically to such evidence. Importantly, the admissibility theory developed
in this paper does not rely on any globally fixed residual scoring rule. Rather,
the main result shows that admissibility is fundamentally driven by the underlying
covariance-adapted residual geometry together with the local monotone manner
in which evidence accumulates through the residual statistics.


Let
\[
h_{tj}^{(i_1,\dots,i_{t-1})}:[0,\infty)\to\mathbb{R}
\]
be any strictly increasing function, where the functional form may locally depend on the
stage $t$, the coordinate $j$, and the previously eliminated indices $(i_1,\dots,i_{t-1})$. For each stage $t$ and each active index $j$, define the generalized residual statistic
\begin{align}
	S_{tj}^{(i_1,\dots,i_{t-1})}(\boldsymbol{X})
	=
	h_{tj}^{(i_1,\dots,i_{t-1})}
	\left(
	\left|
	U_{tj}^{(i_1,\dots,i_{t-1})}(\boldsymbol{X})
	\right|
	\right),
	\label{GENERAL_STATISTIC_DEFN}
\end{align}
for $t,j=1,\dots,n,$ $1\leqslant i_{1}\neq\dots\neq i_{t-1}\leqslant n$ and $i_{l} \neq j$ for all $l=1,\dots,t-1$.

Importantly, the transformation functions are allowed to vary locally with the
stage, the coordinate under consideration, and the previously eliminated indices.
Thus, the proposed framework permits highly heterogeneous and adaptive residual
scoring mechanisms rather than a single global monotone transformation applied
uniformly across all stages and coordinates. In particular, the framework permits stage-dependent, coordinate-dependent,
and history-dependent residual scoring mechanisms within a unified admissible
sequential formulation. When the common stage cutoffs are nonincreasing, this
class reduces to the conventional step-down subclass.

For $1\leqslant t\leqslant n$, define the indices
\begin{align}\label{GENERAL_INDEX}
	j_t(\boldsymbol{X})
	=
	\argmax_{
		j\in\{1,\dots,n\}\setminus
		\{j_1(\boldsymbol{X}),\dots,j_{t-1}(\boldsymbol{X})\}
	}
	S_{tj}^{(j_1(\boldsymbol{X}),\dots,j_{t-1}(\boldsymbol{X}))}
	(\boldsymbol{X}).
\end{align}
Since the residual statistics involved are continuously distributed under the assumed Gaussian model, ties occur with probability zero and may therefore be ignored throughout the theoretical development. An arbitrary measurable tie-breaking convention may be adopted on the null-probability tie set without affecting admissibility.

Given positive stage cutoffs $C^{G}_1,\ldots,C^{G}_n$, the corresponding
monotone residual-based sequential procedure operates as follows. When the
cutoffs are nonincreasing, it is a conventional step-down procedure.
At stage $t$, reject the null hypothesis
$H_{0j_t(\boldsymbol{X})}$ if
\[
S_{tj_t(\boldsymbol{X})}^{
	(j_1(\boldsymbol{X}),\dots,j_{t-1}(\boldsymbol{X}))
}
(\boldsymbol{X})
>
C^{G}_t.
\]
Otherwise, stop and accept all remaining null hypotheses. The procedure
continues until an acceptance occurs or all null hypotheses have been rejected.

The MRD procedure is recovered as the special case corresponding to the identity transformations
\[
h_{tj}^{(i_1,\dots,i_{t-1})}(x)=x
\]
for all admissible choices of
$t,j,i_1,\dots,i_{t-1}$.

\begin{remark}
	Note that the indices $j_t(\boldsymbol{X})$ and $\widetilde{j}_{t}(\boldsymbol{X})$ defined in
	\eqref{GENERAL_INDEX} and \eqref{MRD_INDEX}, respectively, need not coincide
	in general. Thus, the generalized monotone residual-based procedure considered
	here is not simply a reparameterization of the MRD procedure, but may induce a
	substantially different stagewise ordering of the active hypotheses. Even globally common monotone residual transformations need not preserve the future adaptive residual systems, since differences in the stagewise threshold structure may recursively alter the adaptive conditioning histories. This makes
	the admissibility analysis considerably more delicate, since the proof cannot rely
	on an exact correspondence between the stagewise rejection sequences generated
	by the two procedures. 
\end{remark}

\begin{remark}
The class of procedures introduced above preserves the fundamental monotone relationship between the magnitudes of the residual statistics and the 
corresponding stagewise ordering of the active hypotheses. The main result of the next section shows that this local monotone structural property is sufficient to guarantee admissibility under the vector loss formulation considered in this paper.
\end{remark}

%

\section{Admissibility of Monotone Residual-based Step-Down Procedures for Two-Sided Testing}
\label{SECTION_TWO_SIDED_ADMISSIBILITY}

In this section, we establish the main admissibility result of the paper and provide a decision theoretic justification for the class of monotone
residual-based step-down procedures introduced in the previous section. The central question is whether admissibility persists under such highly
general locally adaptive residual transformations despite the complicated data-dependent geometry induced by the adaptive stagewise ordering
mechanism.

Recall that any multiple testing procedure
$\boldsymbol{\Phi}(\boldsymbol{x})=(\phi_1(\boldsymbol{x}),\dots,\phi_n(\boldsymbol{x}))$
induces an individual test function $\phi_i(\boldsymbol{x})$ for testing
$H_{0i}$ against $H_{Ai}$, where $\phi_i(\boldsymbol{x})$ denotes the probability
of rejecting the $i$-th null hypothesis when the observation
$\boldsymbol{X}=\boldsymbol{x}$ is realized. We consider the standard $0-1$ loss function
corresponding to $\phi_i$, given by
\begin{equation}\label{INDIVIDUAL_LOSS}
	L_i\big(\phi_i(\boldsymbol{X}),\boldsymbol{\theta}\big)
	=
	\mathbf{1}\{\theta_i=0\}\phi_i(\boldsymbol{X})
	+
	\mathbf{1}\{\theta_i\neq0\}
	\big(1-\phi_i(\boldsymbol{X})\big),
\end{equation}
while the corresponding risk function is given by
\begin{equation}\label{INDIVIDUAL_RISK}
	R_i\big(\phi_i,\boldsymbol{\theta}\big)
	=
	\mathbf{1}\{\theta_i=0\}
	\mathbb{E}_{\boldsymbol{\theta}}\big(\phi_i(\boldsymbol{X})\big)
	+
	\mathbf{1}\{\theta_i\neq0\}
	\mathbb{E}_{\boldsymbol{\theta}}
	\big(1-\phi_i(\boldsymbol{X})\big).	
\end{equation}

We consider the overall loss function for the multiple testing procedure
$\boldsymbol{\Phi}(\boldsymbol{X})$ to be the vector loss
\begin{equation}\label{VECTOR_LOSS}
	\boldsymbol{L}\big(\boldsymbol{\Phi}(\boldsymbol{X}),\boldsymbol{\theta}\big)
	=
	\big(
	L_1\big(\phi_1(\boldsymbol{X}),\boldsymbol{\theta}\big),
	\dots,
	L_n\big(\phi_n(\boldsymbol{X}),\boldsymbol{\theta}\big)
	\big),
\end{equation}
with corresponding vector risk function
\begin{equation}\label{VECTOR_RISK}
	\boldsymbol{R}\big(\boldsymbol{\Phi},\boldsymbol{\theta}\big)
	=
	\big(
	R_1\big(\phi_1,\boldsymbol{\theta}\big),
	\dots,
	R_n\big(\phi_n,\boldsymbol{\theta}\big)
	\big).	
\end{equation}

A multiple testing procedure $\boldsymbol{\Phi}(\boldsymbol{X})$ is said to be inadmissible
with respect to the vector loss function (\ref{VECTOR_LOSS}) if there exists
another multiple testing procedure $\boldsymbol{\Phi}^{*}(\boldsymbol{X})$ such that
$R_i\big(\phi^{*}_i,\boldsymbol{\theta}\big)
\leq
R_i\big(\phi_i,\boldsymbol{\theta}\big)$
for all $i=1,\dots,n$ and all $\boldsymbol{\theta}\in\mathbb{R}^n$,
with strict inequality holding for at least one $i$ and some
$\boldsymbol{\theta}\in\mathbb{R}^n$. A multiple testing procedure is said to be
admissible if it is not inadmissible in the aforesaid sense. A vector-risk dominator also dominates under additive total misclassification
losses with strictly positive coordinate weights: at every parameter value,
its total risk is no larger, and the strict componentwise improvement at some
parameter value makes the weighted total risk strictly smaller there. More
generally, vector-risk domination weakly decreases the expected numbers of
type I and type II errors. A merely non-decreasing function of those expected
counts need not decrease strictly, however, so inadmissibility for an
unspecified nonlinear count-based risk does not follow without an additional
strictness condition.

In this context, it is worth recalling that
\citet{COHEN_KOL_SACK_2007},
\citet{COHEN_SACK_2005_B},
\citet{COHEN_SACK_2007} and
\citet{COHEN_SACK_2008}
showed that in many common applications involving dependent test statistics,
typical $p$-value-based stepwise testing procedures, including the BH method,
are inadmissible with respect to the vector loss function
(\ref{VECTOR_LOSS}). In particular, the corresponding dominators also
improve any strictly positively weighted additive total misclassification
risk. This is an undesirable feature of many traditional stepwise multiple
testing procedures under dependence.

We now establish vector-loss admissibility for the monotone residual-based
sequential procedures introduced in the previous section. The conclusion
applies to their conventional step-down subclass as well. It is easy to see using standard contrapositive arguments that
a multiple testing procedure is admissible with respect to the vector loss
function (\ref{VECTOR_LOSS}) if and only if each induced individual testing
procedure is admissible with respect to the standard $0-1$ loss
(\ref{INDIVIDUAL_LOSS}). Therefore, in order to establish admissibility of
the proposed methodology, it suffices to show that the induced test for
testing $H_{01}$ against $H_{A1}$ is admissible under the assumed setup.
As in \citet{COHEN_SACK_XU_2009}, our proof relies on a classical result
due to \citet{MAT_TRU_1967}, which provides a necessary and sufficient
condition for admissibility of a test of $H_{01}$ versus $H_{A1}$ when
the joint distribution of $\boldsymbol{X}$ belongs to an exponential family.

We emphasize that admissibility of the generalized monotone residual-based
step-down procedures considered here does not follow as an immediate
consequence of the admissibility of the MRD procedure established in
\citet{COHEN_SACK_XU_2009}. Indeed, for each stage $t$, the statistics
$S_{tj}^{(j_1(\boldsymbol{x}),\dots,j_{t-1}(\boldsymbol{x}))}(\boldsymbol{x})$ depend on the adaptive stagewise
index sequence $j_1(\boldsymbol{x}),\dots,j_{t-1}(\boldsymbol{x})$, each of which itself
depends on the observed data vector $\boldsymbol{x}$. Consequently, the quantities
$\sigma_{j\cdot(j_1(\boldsymbol{x}),\dots,j_{t-1}(\boldsymbol{x}))}$ and the associated
acceptance regions possess highly nontrivial data-dependent geometry.
Moreover, as noted earlier, the stagewise rejection indices generated by
the generalized monotone residual-based procedure need not coincide with
those of the MRD procedure. Even under globally common monotone residual transformations, differences in the stagewise threshold structure may recursively alter the adaptive conditioning sets, thereby inducing substantially different future residual systems and rejection paths. This makes the admissibility analysis
considerably more delicate, since the proof cannot rely on an exact
correspondence between the stagewise rejection sequences generated by the
two procedures. Instead, a more careful structural analysis of the induced
acceptance regions becomes necessary. In the process, we establish a broader structural admissibility result
showing that, for the multiple testing problem
(\ref{TESTING_PROBLEM}), any step-down multiple testing procedure based
on statistics $S_{tj}$, where each $S_{tj}$ is obtained through a
locally adaptive strictly increasing transformation of the absolute value
of the corresponding MRD statistic $U_{tj}$, is admissible with respect
to the vector loss function (\ref{VECTOR_LOSS}). To the best of our knowledge, this general structural
result appears to be new in the literature and substantially extends part
of the admissibility theory developed in
\citet{COHEN_SACK_XU_2009}.

Let $\phi_j(\boldsymbol{x})$ denote the test function induced by the generalized
monotone residual-based step-down procedure for testing
$H_{0j}$ versus $H_{Aj}$ when the observation $\boldsymbol{X}=\boldsymbol{x}$ is realized.
We use the following Matthes--Truax criterion for the Gaussian location
model considered here. It concerns the scalar test of $H_{01}:\theta_1=0$
versus $H_{A1}:\theta_1\neq0$, with $\boldsymbol{\theta}\in\mathbb{R}^n$,
known positive-definite $\boldsymbol{\Sigma}$, and the componentwise
$0$--$1$ loss in \eqref{INDIVIDUAL_LOSS}. In the canonical coordinates
$\boldsymbol{Y}=\boldsymbol{\Sigma}^{-1}\boldsymbol{X}$, the criterion
characterizes admissibility through the acceptance sections parallel to the
$y_1$-axis.

\begin{lem}\label{LEM_MATH_TRUAX}
A test $\phi(\boldsymbol{y})$ of $H_{01}$ versus $H_{A1}$ is admissible
under this scalar $0$--$1$ loss if and only if, for almost every fixed
$(y_2,\ldots,y_n)$, its acceptance section in $y_1$ is an interval, up to a
Lebesgue-null set.
\end{lem}
\begin{proof}
This is the Gaussian special case of the Matthes--Truax criterion
\citep{MAT_TRU_1967}; see also its use for residual-based testing in
\citet{COHEN_SACK_XU_2009}.
\end{proof}


Lemma \ref{LEM_MATH_TRUAX} reduces the admissibility problem to a geometric 
analysis of the induced acceptance region along suitable one-dimensional 
perturbation directions. The remainder of the argument therefore focuses on 
understanding how the adaptive residual-based step-down mechanism evolves 
along such perturbation paths.

Note that, for each fixed $(y_2,\dots,y_n)$, to study the test function $\phi(\boldsymbol{y})=\phi_1(\boldsymbol{x})$ as $y_1$ varies, it would be enough to consider sample points of the form $\boldsymbol{x}+r\boldsymbol{g}$, where $\boldsymbol{g}$ is the first column of $\boldsymbol{\Sigma}$ and $r$ varies. This is true, since $\boldsymbol{y}$ is a function of $\boldsymbol{x}$, and so $\boldsymbol{y}$ evaluated at $\boldsymbol{x}+r\boldsymbol{g}$ is
\begin{eqnarray}
	\boldsymbol{\Sigma}^{-1}(\boldsymbol{x}+r\boldsymbol{g})
	&=&
	\boldsymbol{y}+(r,0,\dots,0)
	=
	(y_1+r,y_2,\dots,y_n).
	\nonumber
\end{eqnarray}
Thus, varying $r$ changes only the first coordinate of $\boldsymbol{y}$ while keeping the remaining coordinates fixed. Lemma \ref{LEM_MATH_TRUAX} therefore reduces the admissibility problem to showing that, along each such line, the acceptance region of the induced test for $H_{01}$ is an interval.

\begin{lem}\label{LEM_SHIFT_PROPERTY}
	The functions $U_{tj}$ as given in (\ref{MRD_STATISTICS}) have the following properties.
	
	For $t \in \{1,\dots,n\}$ and for $j_1,\dots,j_{t-1} \in \{2,\dots,n\}$ with $j_i \neq j_{k} $ for $i \neq k$,
	\begin{eqnarray}
		U^{(j_1,\dots,j_{t-1})}_{t1}(\boldsymbol{x} + r\boldsymbol{g})
		&=&
		U^{(j_1,\dots,j_{t-1})}_{t1}(\boldsymbol{x})
		+
		r\sigma^{\frac{1}{2}}_{1\cdot( j_1,\dots,j_{t-1})}.
		\nonumber
	\end{eqnarray}
	Moreover, for $t \in \{1,\dots,n\}$ and for $j \in \{2,\dots,n\} \setminus \{j_1,\dots,j_{t-1}\},$
	\begin{eqnarray}
		U^{(j_1,\dots,j_{t-1})}_{tj}(\boldsymbol{x} + r\boldsymbol{g})
		&=&
		U^{(j_1,\dots,j_{t-1})}_{tj}(\boldsymbol{x}).
		\nonumber
	\end{eqnarray}
\end{lem}
\begin{proof}
	See Appendix. 
\end{proof}

Lemma \ref{LEM_SHIFT_PROPERTY} shows that, along the direction $\boldsymbol{g}$, the residual statistic corresponding to the first coordinate changes linearly in $r$, whereas the residual statistics corresponding to the remaining active coordinates remain invariant. This separation property is the basic geometric mechanism behind the admissibility argument.

\begin{remark}\label{REMARK_MONOTONE_STRUCTURE}
	Since $\sigma_{1\cdot(j_1,\dots,j_{t-1})}>0$, it follows from Lemma \ref{LEM_SHIFT_PROPERTY} that for each fixed $\boldsymbol{x}\in\mathbb{R}^n$ and given any $(t-1)$ many indices $(j_1,\dots,j_{t-1})$,
	$U^{(j_1,\dots,j_{t-1})}_{t1}(\boldsymbol{x}+r\boldsymbol{g})$
	is strictly increasing in $r$. Consequently,
	$|U^{(j_1,\dots,j_{t-1})}_{t1}(\boldsymbol{x}+r\boldsymbol{g})|$
	initially decreases and then increases as $r$ increases. Therefore, when $|U^{(j_1,\dots,j_{t-1})}_{t1}(\boldsymbol{x}+r\boldsymbol{g})|$
	decreases in $r$, $U^{(j_1,\dots,j_{t-1})}_{t1}(\boldsymbol{x}+r\boldsymbol{g})$
	is negative, while when
	$|U^{(j_1,\dots,j_{t-1})}_{t1}(\boldsymbol{x}+r\boldsymbol{g})|$
	increases in $r$,
	$U^{(j_1,\dots,j_{t-1})}_{t1}(\boldsymbol{x}+r\boldsymbol{g})$
	is positive. This observation is crucial for deriving the interval structure of the induced acceptance region.
\end{remark}

\begin{cor}\label{COR_RESIDUAL_INVARIANCE}
	For any $r \in \mathbb{R}$, we have
	\begin{eqnarray}
		U_{1j}(\boldsymbol{x} + r\boldsymbol{g})
		=
		U_{1j}(\boldsymbol{x})
		\mbox{ for all } j \in \{2,\dots,n\},
		\nonumber
	\end{eqnarray}
	which, in turn, implies the following:
	\begin{eqnarray}
		S_{1j}(\boldsymbol{x} + r\boldsymbol{g})
		=
		S_{1j}(\boldsymbol{x})
		\mbox{ for all } j \in \{2,\dots,n\}.
		\nonumber
	\end{eqnarray}
\end{cor}

In other words, Corollary~\ref{COR_RESIDUAL_INVARIANCE} shows that perturbations along the 
admissibility direction $g$ leave all residual scores corresponding to the 
competing coordinates unchanged, while only the residual statistic 
	associated with the first coordinate evolves along the $g$-direction. 
	Consequently, the induced geometry along the admissibility direction is 
	governed by a single evolving residual component against a collection of 
	geometrically invariant competing residual scores.

In Lemma 3.2 of \citet{COHEN_SACK_XU_2009}, the term $\sigma^{1/2}_{1\cdot(j_1,\dots,j_{t-1})}$ was omitted, most likely due to a typographical oversight, and its role was not explicitly analyzed in the subsequent theoretical argument. In the present setting, this term requires careful attention since $\sigma_{1\cdot(j_1(\boldsymbol{x}),\dots,j_{t-1}(\boldsymbol{x}))}$ depends on the adaptive stagewise indices $j_1(\boldsymbol{x}),\dots,j_{t-1}(\boldsymbol{x})$, each of which is itself a function of the observed data vector $\boldsymbol{x}$. A major difficulty here is that both the residual statistics and the associated
conditional variance structure evolve adaptively with the previously rejected
indices $j_1(\boldsymbol{x}),\dots,j_{t-1}(\boldsymbol{x})$. Consequently, the local geometry of the
step-down procedure is itself data-dependent and potentially unstable under perturbations. It is therefore necessary to understand how these adaptive
indices behave as $\boldsymbol{x}$ varies along the admissibility direction $\boldsymbol g$. The following lemma shows that, remarkably, perturbations along the special direction $\boldsymbol{g}$ preserve the common portion of the stage-wise elimination sequence prior to rejection of $H_{01}$. This local invariance principle partially freezes the adaptive combinatorial structure of the procedure and initiates the reduction of the analysis to a one-dimensional monotone evolution along the $g$-trajectory.


Suppose $\phi_1 (\boldsymbol{x}^{*})=0$ when $\boldsymbol{x}^{*}$ is observed, that is, $\boldsymbol{x}^{*}$ is an acceptance point of $H_{01}$. Then the testing procedure must stop before $H_{01}$ gets rejected. Suppose the testing procedure stops at some stage $t$ without rejecting $H_{01}$. Let $\boldsymbol{x}^{*}+r_{0}\boldsymbol{g}$ be a point of rejection of $H_{01}$, that is,
$\phi_1(\boldsymbol{x}^{*}+r_{0}\boldsymbol{g})=1$, with $r_{0}\neq 0$. Let the testing procedure reject $H_{01}$ at some stage $t_0$ when $\boldsymbol{x}^{*}+r_{0}\boldsymbol{g}$ is observed. The next lemma gives an important identity between the set of indices $j_l(\boldsymbol{x}^{*}+r_{0}\boldsymbol{g})$ and $j_l(\boldsymbol{x}^{*})$, for $1\leqslant l\leqslant \min\{t_0,t\}-1$, which shows that these indices remain invariant whenever $\min\{t,t_0\}>1$.

\begin{lem}\label{LEM_INDEX_INVARIANCE_PARTIAL}
	Under the conditions $\phi_1 (\boldsymbol{x}^{*}) =0$ and $\phi_1 (\boldsymbol{x}^{*}+r_{0}\boldsymbol{g}) = 1$, with $r_{0}\neq 0$, the following holds when $t >1$ and $t_0 >1$:
	\begin{equation}
		j_l(\boldsymbol{x}^{*}+r_{0}\boldsymbol{g})= j_l(\boldsymbol{x}^{*}) \mbox{ for all } l=1,\dots,\min\{t_0,t\} -1, \nonumber
	\end{equation}
where $t_0$ and $t$ are defined as before. 	
\end{lem}

\begin{proof}
	See Appendix.
\end{proof}

Lemma~\ref{LEM_INDEX_INVARIANCE_PARTIAL} shows that, along the admissibility direction $\boldsymbol g$, 
the adaptive stage-wise elimination sequence among the competing coordinates 
remains invariant up to stage $\min\{t,t_0\}-1$. In other words, perturbations 
along the $\boldsymbol g$-direction affect only the residual statistic associated 
with the first coordinate, while the relative ordering among the remaining active 
coordinates remains unchanged throughout the common portion of the adaptive 
rejection path. Consequently, the otherwise highly adaptive multistage testing 
procedure reduces locally to a partially frozen one-dimensional monotone evolution 
along the $\boldsymbol g$-trajectory. This local invariance property forms the key 
structural mechanism underlying the admissibility argument developed below.



Lemma \ref{LEM_INDEX_INVARIANCE_PARTIAL}, coupled with Corollary \ref{COR_RESIDUAL_INVARIANCE}, leads to the following important result on the relation between $t_0$ and $t$ defined before.

\begin{lem}\label{LEM_STOPPING_STAGE}
	Under the conditions $\phi_1 (\boldsymbol{x}^{*}) =0$ and $\phi_1 (\boldsymbol{x}^{*}+r_{0}\boldsymbol{g}) = 1$, with $r_{0}\neq 0$, the monotone residual-based step-down procedure must reject $H_{01}$ within $t$ steps when $\boldsymbol{x}^{*}+r_{0}\boldsymbol{g}$ is observed, that is, $1\leqslant t_0 \leqslant t,$ where $t_0$ and $t$ are defined as before. 
\end{lem}

\begin{proof}
	See Appendix.
\end{proof}

Combining Lemma~\ref{LEM_INDEX_INVARIANCE_PARTIAL} and Lemma~\ref{LEM_STOPPING_STAGE}, the following invariance principle becomes immediate. This invariance principle freezes the adaptive combinatorial structure of the procedure throughout all stages preceding rejection of $H_{01}$ and reduces the analysis locally to a one-dimensional monotone evolution along the $\boldsymbol g$-trajectory.
\begin{cor}\label{COR_INDEX_INVARIANCE}
Under the framework of Lemma~\ref{LEM_STOPPING_STAGE}, we have
\begin{equation}
j_l(\boldsymbol{x}^{*}+r_{0}\boldsymbol{g})= j_l(\boldsymbol{x}^{*}) \mbox{ for all } l=1,\dots,t_0-1. \nonumber
\end{equation}
\end{cor}

Lemma~\ref{LEM_STOPPING_STAGE} establishes that, once a perturbation along the admissibility 
direction $g$ causes the testing procedure to reject $H_{01}$, the rejection 
cannot occur strictly later than the stage at which the original procedure 
stopped without rejecting $H_{01}$. Thus, movement along the $g$-direction 
cannot delay the rejection of $H_{01}$. This property plays a crucial role 
in controlling the evolution of the adaptive rejection path and prevents the 
acceptance region from developing disconnected geometric structures along 
the admissibility direction.

\begin{lem}\label{LEM_INTERVAL_PROPERTY}
	Suppose that for some $\boldsymbol{x}^{*}$ and $r_0 > 0,$ $\phi_1(\boldsymbol{x}^{*})=0$ and $\phi_1(\boldsymbol{x}^{*}+r_0\boldsymbol{g})=1.$ Then $\phi_1(\boldsymbol{x}^{*}+r\boldsymbol{g})=1$ for all $r > r_0.$
\end{lem}

\begin{proof}
	See Appendix.
\end{proof}

Lemma~\ref{LEM_INTERVAL_PROPERTY} establishes a monotonicity property of the induced decision rule 
along the admissibility direction $g$. Specifically, once a perturbation 
$\boldsymbol{x}^{*} + r_0\boldsymbol{g}$ becomes a rejection point for $H_{01}$, every further 
perturbation in the same direction must also remain a rejection point. 
Consequently, the acceptance region along the $\boldsymbol{g}$-direction must necessarily 
have an interval structure. This monotone geometric behavior is precisely 
the key condition required by the Matthes--Truax characterization of 
admissibility.

Using Lemma \ref{LEM_MATH_TRUAX} and Lemma \ref{LEM_INTERVAL_PROPERTY} above, it follows that for testing $H_{01} \mbox { vs } H_{A1}$, the individual decision $\phi_1(\boldsymbol{X})$ induced by the monotone residual-based step-down procedure is admissible with respect to the standard $0-1$ loss (\ref{INDIVIDUAL_LOSS}). The proof that the other tests induced by the procedure for the remaining individual testing problems are admissible follows analogously. Since admissibility of each individual induced decision implies admissibility of the corresponding multiple testing procedure under the vector loss (\ref{VECTOR_LOSS}), we obtain the desired admissibility property in the following theorem.

\begin{thm}\label{THM_ADMISSIBILITY_GENERAL}
	
Suppose $\boldsymbol{X}\sim N_{n}(\boldsymbol{\theta},\boldsymbol{\Sigma})$, where $\boldsymbol{\theta}\in\mathbb{R}^{n}$ is unknown, but fixed and $\boldsymbol{\Sigma}$ is an $n \times n$ arbitrary but known positive definite covariance matrix. Then, for the two-sided multiple testing problem (\ref{TESTING_PROBLEM}), every residual-based sequential procedure that, at each stage, ranks the active statistics $S_{tj}$ and compares their maximum with a common positive stage-specific cutoff, where each $S_{tj}$ is obtained through a locally adaptive strictly increasing transformation of the absolute value of the corresponding MRD statistic $U_{tj}$, is admissible with respect to the vector loss function \eqref{VECTOR_LOSS}. The cutoffs need not be nonincreasing; nonincreasing cutoffs give the conventional step-down subclass.
\end{thm}

Theorem~\ref{THM_ADMISSIBILITY_GENERAL} identifies a broad structural mechanism underlying admissibility, showing that the admissibility property is fundamentally driven by the local monotone ordering structure induced by
the MRD residual statistics rather than by the precise functional form of the residual scoring rule itself. In particular, the admissibility property of the MRD procedure due to \citet{COHEN_SACK_XU_2009} follows immediately
by taking $S_{tj}=|U_{tj}|$. 

Importantly, the admissibility proof does not require the generalized, covariance-adaptive monotone residual-based procedure to preserve either the stagewise rejection ordering or the threshold structure of the original MRD procedure. Moreover, the theorem does not require the generalized residual statistics to arise from a single globally defined transformation of the MRD statistics. The admissibility result continues to hold even when the residual scoring
functions vary locally across stages, coordinates, and adaptive rejection histories. Consequently, Theorem
\ref{THM_ADMISSIBILITY_GENERAL} substantially enlarges the currently known class of admissible covariance-adaptive residual-based procedures.

From a geometric viewpoint, the residual-based construction underlying the MRD procedure may be interpreted as a sequential covariance-adapted orthogonalization mechanism. Along the admissibility directions arising from
the Matthes--Truax characterization \citep{MAT_TRU_1967}, the residual corresponding to the coordinate under consideration evolves monotonically while the competing
residual statistics remain invariant. This effectively reduces the high-dimensional adaptive step-down problem locally to a one-dimensional ordered geometric structure and induces an interval-type acceptance geometry which ultimately drives the admissibility property. In this sense,
Theorem~\ref{THM_ADMISSIBILITY_GENERAL} reveals that admissibility under arbitrary covariance dependence is fundamentally a structural consequence of the covariance-adapted residual geometry underlying the MRD construction.

\begin{remark}
	The Bayesian Step-Down (BSD) procedure introduced by
	\citet{GC_BSD2015} and further developed under a Bayesian model-pursuit
	formulation by \citet{GhoshChakrabartiBSD2026} provides a nontrivial
	example covered by Theorem~\ref{THM_ADMISSIBILITY_GENERAL}. Under the
	prior specifications considered in these works, its stagewise posterior
	scores can be expressed as locally adaptive, strictly increasing
	transformations of the absolute MRD residuals, with the transformations
	allowed to depend on the stage, active coordinate, and rejection
	history. Theorem~\ref{THM_ADMISSIBILITY_GENERAL} therefore implies
	admissibility of the BSD procedure under the vector loss.
\end{remark}

\begin{remark}
	The proof of Theorem~\ref{THM_ADMISSIBILITY_GENERAL} uses only the cutoff
at the current stage and requires it to be common to the active hypotheses;
it does not compare cutoffs across stages. The theorem therefore covers
arbitrary positive stage-specific cutoffs. The following bivariate example
illustrates this point. Its cutoffs increase across stages, so it is a
residual-based sequential rule rather than a conventional step-down rule.
	
	Let
	\[
	\boldsymbol{\Sigma}=
	\begin{pmatrix}
		1&\rho\\
		\rho&1
	\end{pmatrix},
	\qquad \rho=\frac12,
	\qquad C_1=1,\qquad C_2=2,
	\]
	and use the identity score transformation. The procedure is defined
	as follows. At stage 1, compute
	\[
	U_{11}(\boldsymbol{X})=\frac{X_1-\rho X_2}{s},
	\qquad
	U_{12}(\boldsymbol{X})=\frac{X_2-\rho X_1}{s},
	\qquad
	s=\sqrt{1-\rho^2}=\frac{\sqrt{3}}{2},
	\]
	and let
	\[
	j_1=\arg\max_{j\in\{1,2\}}|U_{1j}(\boldsymbol{X})|.
	\]
	If $|U_{1j_1}(\boldsymbol{X})|\leq C_1$, the procedure stops and
	accepts both hypotheses. If $|U_{1j_1}(\boldsymbol{X})|>C_1$, it
	rejects $H_{0j_1}$ and proceeds to stage 2. Let $j_2$ be the remaining
	index. Since only coordinate $j_2$ is then active and its variance is
	one, its stage-2 residual is
	\[
	U_{2j_2}^{(j_1)}(\boldsymbol{X})=X_{j_2}.
	\]
	The procedure rejects $H_{0j_2}$ if
	$|U_{2j_2}^{(j_1)}(\boldsymbol{X})|>C_2$; otherwise, it stops and accepts
	$H_{0j_2}$.
	
	For illustration, take
	$\boldsymbol{\theta}=(5/2,0)$ and suppose that the observed value is
	$\boldsymbol{x}=(5/2,19/10)$. The first-stage residuals are
	\[
	U_{11}(\boldsymbol{x})
	=\frac{x_1-\rho x_2}{s}
	=\frac{31}{10\sqrt{3}}\approx 1.790,
	\qquad
	U_{12}(\boldsymbol{x})
	=\frac{x_2-\rho x_1}{s}
	=\frac{13}{10\sqrt{3}}\approx 0.751.
	\]
	Thus $j_1=1$ and $|U_{11}(\boldsymbol{x})|>C_1$, so the procedure
	rejects $H_{01}$ and proceeds to stage 2. There,
	\[
	U_{2,2}^{(1)}(\boldsymbol{x})=x_2=\frac{19}{10}<C_2,
	\]
	so it stops and accepts $H_{02}$. The specified parameter and
	observation illustrate this realized path; admissibility follows from
	the acceptance-section argument below, not from this single
	observation.
	
	To verify the interval property, write
	$\boldsymbol{Y}=\boldsymbol{\Sigma}^{-1}\boldsymbol{X}$. Then $U_i=sY_i$, so fixing $U_2$ is
	equivalent to fixing $Y_2$, and an interval in $U_1$ corresponds to an
	interval in $Y_1$. Also,
	\[
	X_1=Y_1+\rho Y_2=\frac{U_1+\rho U_2}{s}.
	\]
	For fixed $U_2=u_2$, the acceptance region of the test for $H_{01}$ is,
	apart from tie-boundary points,
	\[
	\mathcal{A}_1(u_2)=
	\begin{cases}
		[-1,1], & |u_2|\leq 1,\\[1mm]
		(-|u_2|,|u_2|)
		\cap
		[-\rho u_2-sC_2,\,-\rho u_2+sC_2],
		& |u_2|>1.
	\end{cases}
	\]
	When $|u_2|\leq 1$, the procedure accepts $H_{01}$ at stage 1 exactly
	when $|u_1|\leq 1$. When $|u_2|>1$, $H_{01}$ must have the smaller
	first-stage residual magnitude to remain active after stage 1.
	Following rejection of $H_{02}$, it is accepted at stage 2 precisely
	when $|X_1|\leq C_2$. Since $\rho=1/2$ and $sC_2=\sqrt{3}$, the
	second line is
	\[
	(-|u_2|,|u_2|)
	\cap
	[-u_2/2-\sqrt{3},\,-u_2/2+\sqrt{3}],
	\]
	which is an interval or the empty set. By symmetry, the same interval
	property holds for the acceptance section of the test of $H_{02}$.
	
	The Matthes--Truax interval criterion~\citep{MAT_TRU_1967} therefore
	gives admissibility of each induced test under its scalar $0$--$1$
	loss. Hence, the procedure is admissible under vector loss. This
	example shows that nonincreasing cutoffs are not necessary for the
	admissibility conclusion.
\end{remark}

\section{Admissibility of Monotone Residual-based Step-Down Procedures for One-Sided Testing}
\label{SECTION_ONE_SIDED_ADMISSIBILITY}

Throughout this section, we assume
$\boldsymbol{X}\sim N_n(\boldsymbol{\theta},\boldsymbol{\Sigma})$,
where $\boldsymbol{\Sigma}$ is known and positive definite.

The two-sided MRD testing procedure discussed in Section~\ref{SECTION_TWO_SIDED_ADMISSIBILITY} ranks absolute residuals because departures in either
direction count against the null. For an alternative hypothesis $\theta_i>0$, however,
a large negative residual is evidence in the wrong direction and must not be
treated as strong evidence for rejection. The one-sided MRD rule therefore
retains the signs: it ranks the active signed residuals and compares the largest
one with an upper-tail cutoff, as defined below. This signed version of MRD
was considered by \citet{COHEN_SACK_XU_2009}, who treated both point-null and
composite-null one-sided formulations. Our purpose here is to
identify the admissibility mechanism for this one-sided rule and extend it to
locally adaptive score transformations that may vary with the stage, active
coordinate, and rejection history.

The geometric argument is most transparent in the canonical-statistic
coordinates $\boldsymbol{Y}=\boldsymbol{\Sigma}^{-1}\boldsymbol{X}$.
For $j=1,\ldots,n$, let $\boldsymbol{e}_j\in\mathbb{R}^n$ denote the
$j$th standard basis vector. Fixing $\boldsymbol{Y}_{-j}$ and varying
$Y_j$ corresponds to varying $\boldsymbol{X}$ along
$\boldsymbol{X}+r\boldsymbol{\Sigma}\boldsymbol{e}_j$. For the scalar one-sided testing problem, the Matthes--Truax condition
\citep{MAT_TRU_1967} used here is that, for almost every fixed
$\boldsymbol{Y}_{-j}$, the acceptance section in $Y_j$ be a lower
half-line, up to a null set; equivalently, the rejection decision is
nondecreasing in $Y_j$. While hypothesis $j$ remains active, its signed
residual increases strictly along this line and the competing active
residuals remain fixed. The step-down path argument then shows that this
rejection cannot be lost as $Y_j$ increases, establishing the required
half-line property for the signed MRD procedure.


\subsection{Testing formulations and vector losses}

We distinguish the point-null and composite-null formulations because they
have different parameter spaces and therefore different type-I error regions.
Throughout this section, $\phi_i(\boldsymbol{x})\in[0,1]$ denotes the
probability of rejecting $H_{0i}$ at observation $\boldsymbol{x}$; for a
nonrandomized rule it is the indicator of rejection.

In the point-null formulation, the parameter space and testing problem are
\begin{equation}\label{ONE_SIDED_POINT_PROBLEM}
\begin{gathered}
 \boldsymbol{\theta}\in\Theta_{\mathrm{pt}}=[0,\infty)^n,\\
 H_{0i}:\theta_i=0
 \quad\text{versus}\quad
 H_{Ai}:\theta_i>0,
 \qquad i=1,\ldots,n.
\end{gathered}
\end{equation}
For this problem, the $i$th component has the $0$--$1$ loss and risk
\begin{align}
 L_i^{\mathrm{pt}}\big(\phi_i(\boldsymbol{X}),\boldsymbol{\theta}\big)
 &=\mathbf{1}\{\theta_i=0\}\phi_i(\boldsymbol{X})
   +\mathbf{1}\{\theta_i>0\}
       \big[1-\phi_i(\boldsymbol{X})\big],
       \label{ONE_SIDED_POINT_LOSS}\\
 R_i^{\mathrm{pt}}(\phi_i,\boldsymbol{\theta})
 &=\mathbf{1}\{\theta_i=0\}
       \mathbb{E}_{\boldsymbol{\theta}}\big[\phi_i(\boldsymbol{X})\big]
   +\mathbf{1}\{\theta_i>0\}
       \mathbb{E}_{\boldsymbol{\theta}}\big[1-\phi_i(\boldsymbol{X})\big].
       \label{ONE_SIDED_POINT_RISK}
\end{align}
The corresponding multiple-testing loss and risk functions are the vectors
\begin{align}
 \boldsymbol{L}^{\mathrm{pt}}(\boldsymbol{\Phi},\boldsymbol{\theta})
 &=\big(L_1^{\mathrm{pt}}(\phi_1(\boldsymbol{X}),\boldsymbol{\theta}),
        \ldots,
        L_n^{\mathrm{pt}}(\phi_n(\boldsymbol{X}),\boldsymbol{\theta})\big),
        \label{ONE_SIDED_POINT_VECTOR_LOSS}\\
 \boldsymbol{R}^{\mathrm{pt}}(\boldsymbol{\Phi},\boldsymbol{\theta})
 &=\big(R_1^{\mathrm{pt}}(\phi_1,\boldsymbol{\theta}),\ldots,
        R_n^{\mathrm{pt}}(\phi_n,\boldsymbol{\theta})\big).
        \label{ONE_SIDED_POINT_VECTOR_RISK}
\end{align}

In the composite-null formulation, the parameter space and testing problem
are instead
\begin{equation}\label{ONE_SIDED_COMPOSITE_PROBLEM}
\begin{gathered}
 \boldsymbol{\theta}\in\Theta_{\mathrm{comp}}=\mathbb{R}^n,\\
 H_{0i}:\theta_i\leq 0
 \quad\text{versus}\quad
 H_{Ai}:\theta_i>0,
 \qquad i=1,\ldots,n.
\end{gathered}
\end{equation}
Its componentwise $0$--$1$ loss and risk are
\begin{align}
 L_i^{\mathrm{comp}}\big(\phi_i(\boldsymbol{X}),\boldsymbol{\theta}\big)
 &=\mathbf{1}\{\theta_i\leq 0\}\phi_i(\boldsymbol{X})
   +\mathbf{1}\{\theta_i>0\}
       \big[1-\phi_i(\boldsymbol{X})\big],
       \label{ONE_SIDED_COMPOSITE_LOSS}\\
 R_i^{\mathrm{comp}}(\phi_i,\boldsymbol{\theta})
 &=\mathbf{1}\{\theta_i\leq 0\}
       \mathbb{E}_{\boldsymbol{\theta}}\big[\phi_i(\boldsymbol{X})\big]
   +\mathbf{1}\{\theta_i>0\}
       \mathbb{E}_{\boldsymbol{\theta}}\big[1-\phi_i(\boldsymbol{X})\big].
       \label{ONE_SIDED_COMPOSITE_RISK}
\end{align}
The associated multiple-testing loss and risk functions are given by the vectors
\begin{align}
 \boldsymbol{L}^{\mathrm{comp}}(\boldsymbol{\Phi},\boldsymbol{\theta})
 &=\big(L_1^{\mathrm{comp}}(\phi_1(\boldsymbol{X}),\boldsymbol{\theta}),
        \ldots,
        L_n^{\mathrm{comp}}(\phi_n(\boldsymbol{X}),\boldsymbol{\theta})\big),
        \label{ONE_SIDED_COMPOSITE_VECTOR_LOSS}\\
 \boldsymbol{R}^{\mathrm{comp}}(\boldsymbol{\Phi},\boldsymbol{\theta})
 &=\big(R_1^{\mathrm{comp}}(\phi_1,\boldsymbol{\theta}),\ldots,
        R_n^{\mathrm{comp}}(\phi_n,\boldsymbol{\theta})\big).
        \label{ONE_SIDED_COMPOSITE_VECTOR_RISK}
\end{align}

Admissibility is considered separately on each stated parameter space. In each
case, it means that no competing procedure has every component of its vector
risk no larger throughout that parameter space and a strictly smaller
component risk at some parameter point in it.

\subsection{Signed residual procedures and the admissibility result}

For one-sided alternatives, evidence is directional. In testing
$H_{0i}:\theta_i\leq 0$ against $H_{Ai}:\theta_i>0$, a large negative
residual is not evidence for rejection. The one-sided MRD rule therefore
retains the signs of the residuals and uses an upper-tail comparison.

At stage $t$, let $(i_1,\ldots,i_{t-1})$ be the ordered tuple of previously
rejected coordinates, and let
\[
A_t=\{1,\ldots,n\}\setminus\{i_1,\ldots,i_{t-1}\}
\]
be the active set. The signed MRD rule selects
\[
j_t^+(\boldsymbol{X})
=\arg\max_{j\in A_t}
U_{tj}^{(i_1,\ldots,i_{t-1})}(\boldsymbol{X}).
\]
Given common positive stage cutoffs $C_t^+>0$, it rejects
$H_{0,j_t^+(\boldsymbol{X})}$ if
\[
U_{t,j_t^+(\boldsymbol{X})}^{(i_1,\ldots,i_{t-1})}(\boldsymbol{X})
>C_t^+.
\]
Otherwise, it stops and accepts all hypotheses still active. If the
cutoffs are nonincreasing, this is the conventional signed MRD step-down
rule; the admissibility result below also covers arbitrary positive
stage-specific cutoffs.

For an ordered tuple $(i_1,\ldots,i_{t-1})$ of previously rejected
coordinates and an active coordinate $j$, let
$h_{tj}^{(i_1,\ldots,i_{t-1})}:\mathbb{R}\to\mathbb{R}$ be strictly
increasing. Define the signed-residual score
\begin{equation}\label{SIGNED_MRD_SCORE}
 S_{tj}^{+,(i_1,\ldots,i_{t-1})}(\boldsymbol{x})
 =h_{tj}^{(i_1,\ldots,i_{t-1})}
   \!\left(U_{tj}^{(i_1,\ldots,i_{t-1})}(\boldsymbol{x})\right).
\end{equation}
At stage $t$, select the active coordinate with the largest score. Reject its
null if the score exceeds the common stage cutoff $C_t^G$; otherwise stop and
accept all hypotheses still active. The cutoffs satisfy $C_t^G>0$ at each stage and are common to all active
coordinates. This class contains the signed MRD rule:
take every $h_{tj}^{(i_1,\ldots,i_{t-1})}$ to be the identity and set
$C_t^G=C_t^+$ for every stage $t$.

We use the one-sided Matthes--Truax admissibility criterion
\citep{MAT_TRU_1967}, which also underlies the one-sided MRD analysis of
\citet{COHEN_SACK_XU_2009}. In the natural coordinates
$\boldsymbol{Y}=\boldsymbol{\Sigma}^{-1}\boldsymbol{X}$, the criterion
reduces admissibility of the individual test to the shape of its acceptance
section along a coordinate line.

\begin{lem}
\label{LEM_ONE_SIDED_MATH_TRUAX}
Consider the scalar test induced for coordinate $1$. In the point-null
problem, use the parameter space $\Theta_{\mathrm{pt}}=[0,\infty)^n$ and
the scalar loss $L_1^{\mathrm{pt}}$ in \eqref{ONE_SIDED_POINT_LOSS}; in the
composite-null problem, use $\Theta_{\mathrm{comp}}=\mathbb{R}^n$ and the
scalar loss $L_1^{\mathrm{comp}}$ in \eqref{ONE_SIDED_COMPOSITE_LOSS}. If,
for almost every fixed $(y_2,\ldots,y_n)$, the acceptance section in $y_1$
is a lower half-line up to a Lebesgue-null set, then the induced test is
admissible for the corresponding scalar decision problem.
\end{lem}

\begin{proof}
This is the one-sided sufficient part of the Matthes--Truax acceptance-section
criterion for the Gaussian model and the two parameter-space/loss formulations
stated above; see~\citet{MAT_TRU_1967} and its application to point-null and
composite-null residual testing in~\citet{COHEN_SACK_XU_2009}.
\end{proof}

Lemma~\ref{LEM_ONE_SIDED_MATH_TRUAX} reduces the one-sided admissibility
argument to a geometric condition on each induced coordinate test: with the
other natural coordinates fixed, its acceptance set must be a lower
half-line. For the signed-residual step-down rule, increasing the natural
coordinate associated with the hypothesis under consideration increases its
signed residual, while the competing residuals remain unchanged as long as
that hypothesis is active. The step-down path among the competing hypotheses
therefore remains fixed up to the stage at which the hypothesis is rejected.
Strictly increasing score transformations preserve this ordering. Because the
cutoff is common to all active coordinates at a given stage, if the target
coordinate overtakes the selected competitor, its score also exceeds that
stage's cutoff. Hence the acceptance section has the required lower half-line
form. The theorem applies this argument to the locally adaptive
signed-residual procedures under both one-sided formulations; admissibility
of the induced coordinate tests yields admissibility under the corresponding
vector losses.

\begin{thm}
\label{THM_ONE_SIDED_ADMISSIBILITY}
Suppose $\boldsymbol{X}\sim N_n(\boldsymbol{\theta},\boldsymbol{\Sigma})$ with
known positive definite $\boldsymbol{\Sigma}$. The locally adaptive
signed-residual step-down procedure defined above, with strictly increasing
history-, stage-, and coordinate-dependent score transformations and positive
stage cutoffs, each common to the active coordinates, is admissible under the
corresponding vector risk in each of the following decision problems:
\begin{enumerate}
 \item the point-null problem on $\Theta_{\mathrm{pt}}=[0,\infty)^n$, under
 the componentwise loss vector $\boldsymbol{L}^{\mathrm{pt}}$ in
 \eqref{ONE_SIDED_POINT_VECTOR_LOSS} and its risk vector
 $\boldsymbol{R}^{\mathrm{pt}}$ in
 \eqref{ONE_SIDED_POINT_VECTOR_RISK};
 \item the composite-null problem on $\Theta_{\mathrm{comp}}=\mathbb{R}^n$,
 under the componentwise loss vector $\boldsymbol{L}^{\mathrm{comp}}$ in
 \eqref{ONE_SIDED_COMPOSITE_VECTOR_LOSS} and its risk vector
 $\boldsymbol{R}^{\mathrm{comp}}$ in
 \eqref{ONE_SIDED_COMPOSITE_VECTOR_RISK}.
\end{enumerate}
In each case, admissibility means that no competing procedure has every
component of the corresponding vector risk no larger throughout the stated
parameter space, with a strict improvement in at least one component at some
parameter point in that space.
\end{thm}

\begin{proof}
	See Appendix. 
\end{proof}

The theorem extends the one-sided MRD admissibility result to locally adaptive
scores that may vary across stages, coordinates, and previously rejected
indices. It applies under arbitrary known positive definite covariance, with
only strict monotonicity of the score in its signed residual and positive
stage cutoffs common to the active coordinates. The conclusion is componentwise
vector-risk admissibility in each explicitly stated model; it does not assert
admissibility under additive total loss or any FDR-control guarantee. The proof
is given in the Appendix.

\begin{remark}
When the positive stage cutoffs are nonincreasing, the procedure is a
conventional step-down rule. The theorem also covers arbitrary positive
stage-specific cutoffs, provided that each cutoff is common to the active
coordinates. For a nonmonotone cutoff sequence, the rule is a residual-based
sequential procedure rather than a conventional step-down procedure.
\end{remark}

\section{Bayesian Approximation under Additive Total Loss}
\label{SECTION_BAYES_LIMITS}

The admissibility results discussed in Sections~\ref{SECTION_TWO_SIDED_ADMISSIBILITY} and \ref{SECTION_ONE_SIDED_ADMISSIBILITY} concern vector loss, under which the
component risks are compared separately. They do not, however, settle admissibility
under additive total loss, which sums the componentwise misclassification
losses. Under this aggregate criterion, a competing procedure may increase
risk for some hypotheses while reducing it for others, and admissibility
depends on the resulting total-risk difference.

Dependence makes this comparison particularly difficult for MRD
procedures. At each stage, the active set and the ordered history of
previous rejections are determined by the data. Along a fixed history,
the residuals of the other active hypotheses remain unchanged as the
observation varies in the coordinate direction for a hypothesis that
remains active. But a change in the observation can alter which
hypothesis is selected or rejected, thereby changing the subsequent
history and the conditioning sets used to form later residuals. Thus,
the one-dimensional acceptance-section argument for vector-loss
admissibility does not control the aggregate risk difference: a
total-loss comparison must account both for cross-hypothesis risk
trade-offs and for the data-dependent changes in the rejection path.

This motivates a related decision-theoretic question: can the residual
step-down procedures studied here be approximated by Bayes rules under the
corresponding additive total losses? Limits of Bayes rules are central to
complete-class theory, although the implications depend on the topology
and regularity conditions involved. The almost-everywhere closure
established below is a specific approximation result; by itself, it
neither proves admissibility under total loss nor characterizes a complete
class. The construction is connected to the preceding admissibility
arguments through the same local residual representation in natural
coordinates, which allows us to encode each finite rejection path by
prior atoms.

For the two-sided problem, define
\[
L_{\mathrm{tot}}(\boldsymbol{\Phi}(\boldsymbol{X}),\boldsymbol{\theta})
=\sum_{j=1}^n L_j(\phi_j(\boldsymbol{X}),\boldsymbol{\theta}),
\qquad
R_{\mathrm{tot}}(\boldsymbol{\Phi},\boldsymbol{\theta})
=\mathbb{E}_{\boldsymbol{\theta}}
L_{\mathrm{tot}}(\boldsymbol{\Phi}(\boldsymbol{X}),\boldsymbol{\theta}).
\]
For the point-null and composite-null one-sided problems, respectively, set
\begin{align*}
	L_{\mathrm{tot}}^{\mathrm{pt}}(\boldsymbol{\Phi}(\boldsymbol{X}),\boldsymbol{\theta})
	&=\sum_{j=1}^n L_j^{\mathrm{pt}}
	(\phi_j(\boldsymbol{X}),\boldsymbol{\theta}),
	&R_{\mathrm{tot}}^{\mathrm{pt}}(\boldsymbol{\Phi},\boldsymbol{\theta})
	&=\mathbb{E}_{\boldsymbol{\theta}}
	L_{\mathrm{tot}}^{\mathrm{pt}}(\boldsymbol{\Phi}(\boldsymbol{X}),\boldsymbol{\theta}),\\
	L_{\mathrm{tot}}^{\mathrm{comp}}(\boldsymbol{\Phi}(\boldsymbol{X}),\boldsymbol{\theta})
	&=\sum_{j=1}^n L_j^{\mathrm{comp}}
	(\phi_j(\boldsymbol{X}),\boldsymbol{\theta}),
	&R_{\mathrm{tot}}^{\mathrm{comp}}(\boldsymbol{\Phi},\boldsymbol{\theta})
	&=\mathbb{E}_{\boldsymbol{\theta}}
	L_{\mathrm{tot}}^{\mathrm{comp}}(\boldsymbol{\Phi}(\boldsymbol{X}),\boldsymbol{\theta}).
\end{align*}
Under any prior, a Bayes rule for one of these additive losses rejects
coordinate $j$ when its posterior probability of belonging
to the alternative exceeds $1/2$; the decisions at equality can be chosen
arbitrarily.

Theorem~\ref{THM_BAYES_LIMITS} in this section gives a Bayesian representation for the same residual
step-down structures considered in the admissibility theorems. The argument
uses the same local residual geometry, but now to construct Bayes rules under
additive total loss. We first state the three ingredients of the construction:
locality of each residual, an encoding of a finite rejection path, and a
finite-support prior whose Bayes rule selects the encoded path.

The first ingredient is a representation of each conditional residual in
natural coordinates. It is this support property—not a comparison of the
full adaptive paths—that makes a stagewise encoding possible.

\begin{lem}\label{LEM_BAYES_RESIDUAL_LOCALITY}
For any history $r=(i_1,\ldots,i_{t-1})$, let
$R=\{i_1,\ldots,i_{t-1}\}$. In the natural coordinates
$\boldsymbol{Y}=\boldsymbol{\Sigma}^{-1}\boldsymbol{X}$, each active
coordinate $j\notin R$ has a vector $\boldsymbol{v}_{j,R}$ such that
\begin{equation}\label{eq:bayes-residual-locality-main}
 U_{tj}^{(r)}(\boldsymbol{X})
 =\boldsymbol{v}_{j,R}^{\mathsf T}\boldsymbol{Y},
 \qquad
 \operatorname{supp}(\boldsymbol{v}_{j,R})\subseteq R\cup\{j\},
 \qquad
 (\boldsymbol{v}_{j,R})_j=\sigma_{j\cdot R}^{1/2}>0.
\end{equation}
Here, for any vector $\boldsymbol v$, its support is
$\operatorname{supp}(\boldsymbol v)=\{k:v_k\ne0\}$.
Consequently, for distinct active coordinates $j$ and $k$, the vectors
$\boldsymbol{v}_{j,R}$ and $\boldsymbol{v}_{k,R}$ are linearly
independent.
\end{lem}

\begin{proof}
	See Appendix. 
\end{proof}

Thus, at a fixed history, each active residual can depend on the
coordinates already rejected and the coordinate currently tested, but
not on any other active coordinate. In particular, the current
coordinate has a nonzero coefficient in its own residual and coefficient
zero in every other active residual. Hence the residual directions
corresponding to two distinct active coordinates are linearly
independent. This locality will distinguish the affine increments when
two paths first select different active coordinates; the weighting
introduced below will then make that first difference determine which
path is preferred.

We develop the required representation first on a bounded observation
set, over which the residuals range in bounded intervals. We approximate
the monotone transformation functions by smooth, strictly increasing
functions. At each history, a common exponential transformation of the
active scores can then be chosen so that the transformed scores are
convex on the relevant intervals. Since this transformation is strictly
increasing, transforming the common cutoff in the same way preserves
both the ordering of the active scores and the stagewise rejection
decision. Each transformed convex score is uniformly approximated by a
finite maximum of affine functions. In the signed one-sided case, the
affine pieces have positive slopes. The Appendix justifies these
approximations and the passage back to the original procedure.

Fix one such finite piecewise-affine approximation. By
Lemma~\ref{LEM_BAYES_RESIDUAL_LOCALITY}, at any history the affine score
for a currently tested coordinate depends only on the coordinates
already rejected and the coordinate under consideration. A path
$p=((j_1,\pi_1),\ldots,(j_q,\pi_q))$ records the coordinates rejected,
in order, together with the affine piece used for the score at each
stage; we also include the empty path. Write
$r_t=(j_1,\ldots,j_{t-1})$ for the ordered rejection history and
$R_t=\{j_1,\ldots,j_{t-1}\}$ for its set of coordinates. Let
$s_{\pi_t}u+c_{\pi_t}$ denote the affine score piece selected at stage
$t$. The following lemma associates an affine function with each path
and establishes the properties needed for the Bayes-limit construction.

\begin{lem}\label{LEM_BAYES_PATH_ENCODING}
For $0<\varepsilon<1$, associate to each such path the affine function
\begin{equation}\label{eq:bayes-path-score-main}
 f_p^\varepsilon(\boldsymbol{y})
 =\sum_{t=1}^q\varepsilon^{t-1}
 \left[
 s_{\pi_t}\boldsymbol{v}_{j_t,R_t}^{\mathsf T}\boldsymbol{y}
 +c_{\pi_t}-T_{r_t}
 \right]
 =(\boldsymbol{a}_p^\varepsilon)^{\mathsf T}\boldsymbol{y}
 +b_p^\varepsilon,
 \qquad f_\varnothing^\varepsilon=0,
\end{equation}
where $T_{r_t}$ is the transformed cutoff at history $r_t$. For all
sufficiently small $\varepsilon$ outside a finite exceptional set, the
path functions are pairwise distinct and their slope vectors satisfy
$\operatorname{supp}(\boldsymbol{a}_p^\varepsilon)=S(p)$ simultaneously
for every path $p$, where $S(p)=\{j_1,\ldots,j_q\}$. In the signed
one-sided case, $(\boldsymbol{a}_p^\varepsilon)_j>0$ for every
$j\in S(p)$. Moreover, for almost every $\boldsymbol{y}$, the unique maximizing path
eventually agrees with the rejection path of the finite piecewise-affine
rule as $\varepsilon\downarrow0$.
\end{lem}

\begin{proof}
	See Appendix. 
\end{proof}

Lemma~\ref{LEM_BAYES_PATH_ENCODING} shows that, outside a null set of observations, the rejection
path is the unique maximizing path for all sufficiently small
$\varepsilon$. At any such observation, the first stage at which two
paths differ contributes a strict gap; the factor
$\varepsilon^{t-1}$ makes that gap dominate the later contributions
when $\varepsilon$ is sufficiently small. The slope vector of each
path function also records exactly the coordinates rejected along that
path. Thus, the lemma represents the sequential decision rule by a
finite collection of affine functions, with each function's slope
identifying its rejection set. The next lemma uses these functions to
construct a finite-support prior.

For fixed $\varepsilon$ and $\kappa$, the path functions specify the
support points and weights of that prior. Under the two-sided
formulation, the support identity makes a coordinate's parameter value
nonzero exactly when the path rejects it. In the signed one-sided
formulations, the positivity property places rejected coordinates in
the alternative and leaves the other coordinates at zero. These are
the parameter labels needed for the corresponding total losses.

\begin{lem}\label{LEM_BAYES_FINITE_PRIOR}
	Fix the finite piecewise-affine step-down rule used in
	Lemma~\ref{LEM_BAYES_PATH_ENCODING}, and let $\mathcal{P}$ be its finite
	collection of rejection paths. Choose $\varepsilon\in(0,1)$ sufficiently
	small and outside the finite exceptional set specified in
	Lemma~\ref{LEM_BAYES_PATH_ENCODING}. For each $p\in\mathcal{P}$, write
	its affine path score as
	\[
	f_p^\varepsilon(\boldsymbol{y})
	=(\boldsymbol{a}_p^\varepsilon)^{\mathsf T}\boldsymbol{y}
	+b_p^\varepsilon,
	\]
	as in \eqref{eq:bayes-path-score-main}. For $\kappa>0$, set
	$\boldsymbol{\Omega}=\boldsymbol{\Sigma}^{-1}$ and associate with path
	$p$ the parameter value
	$\boldsymbol{\theta}_p=\kappa\boldsymbol{a}_p^\varepsilon$ and
	unnormalized weight
	\begin{equation}\label{eq:bayes-path-prior-main}
		w_p=
		\exp\!\left\{
		\kappa b_p^\varepsilon+
		\frac{\kappa^2}{2}
		(\boldsymbol{a}_p^\varepsilon)^{\mathsf T}
		\boldsymbol{\Omega}\boldsymbol{a}_p^\varepsilon
		\right\}.
	\end{equation}
	When multiple paths yield the same parameter value, we assign that value
	the sum of their weights and normalize the resulting weights to obtain a
	proper prior. The resulting total-loss Bayes rule converges almost
	everywhere, as $\kappa\to\infty$, to the decision vector associated with
	the maximizing path. In the one-sided case, the prior is supported on
	$[0,\infty)^n$.
\end{lem}

\begin{proof}
	See Appendix. 
\end{proof}

Different paths may determine the same parameter value. The support
identity in Lemma~\ref{LEM_BAYES_PATH_ENCODING} implies that any such
paths have the same rejected set. In the signed one-sided case, the
positivity property gives the same conclusion about which coordinates
are in the alternative. Thus, combining the prior masses at coincident
parameter values does not combine paths with different null/alternative
labels. Since the total-loss Bayes rule depends on the posterior through
these coordinatewise labels, this aggregation does not change the
decision rule.

Before aggregation, the posterior contribution associated with path
$p$ is proportional to $\exp\{\kappa
f_p^\varepsilon(\boldsymbol y)\}$. Hence, for almost every
$\boldsymbol y$, the posterior concentrates on the label vector of
the maximizing path as $\kappa\to\infty$. This proves the
finite-approximation claim. To obtain the full closure result, it
remains to remove the smoothing and affine-envelope approximations and
to expand the bounded observation set; these limiting steps are
carried out in the proof of Theorem~\ref{THM_BAYES_LIMITS}.

\begin{thm}\label{THM_BAYES_LIMITS}
Suppose $n\geq2$ and $\boldsymbol{\Sigma}$ is known and positive definite.

\begin{enumerate}
 \item For the two-sided problem, every procedure defined by
 \eqref{GENERAL_STATISTIC_DEFN} and common positive stage cutoffs $C_t^G>0$ is the Lebesgue-a.e. limit, in the natural
 coordinates $\boldsymbol{Y}=\boldsymbol{\Sigma}^{-1}\boldsymbol{X}$, of
 Bayes rules under $L_{\mathrm{tot}}$ for a sequence of finite-support
 priors on $\mathbb{R}^n$.

 \item For the signed one-sided class defined by \eqref{SIGNED_MRD_SCORE}
 and common positive stage cutoffs, the same
 conclusion holds under $L_{\mathrm{tot}}^{\mathrm{pt}}$ on
 $\Theta_{\mathrm{pt}}$ and under $L_{\mathrm{tot}}^{\mathrm{comp}}$ on
 $\Theta_{\mathrm{comp}}$. In both cases, the approximating priors may be
 supported on $[0,\infty)^n$.
\end{enumerate}
\end{thm}

\begin{proof}
	See Appendix. 
\end{proof}

The theorem places both procedure classes in the almost-everywhere
closure of finite-support Bayes rules under their respective additive
total losses. The approximating priors depend on the target procedure;
the theorem does not assert that the procedure is Bayes for any one
fixed prior.

The almost-everywhere convergence also implies pointwise convergence
of the risk functions. Every Gaussian sampling distribution has a
strictly positive density, so the convergence holds almost surely
under each fixed parameter value. Since the losses are bounded, the
dominated convergence theorem gives convergence of the risks at every
fixed parameter vector. This is pointwise risk convergence, not
uniform convergence over the parameter space.

Bayes-rule closure is relevant to complete-class theory, but membership
in that closure does not by itself establish admissibility or
characterize a complete class. The theorem therefore provides a
Bayesian representation under total loss without resolving
total-loss admissibility.

%
%

\section{Discussion}\label{SECTION_DISCUSSION}

In this paper, we established general admissibility results for two-sided and one-sided monotone residual-based step-down procedures under arbitrary covariance dependence within a multivariate Gaussian framework. For two-sided testing, admissibility under the vector loss function \eqref{VECTOR_LOSS} follows for a broad class of locally adaptive transformations of absolute residuals. For one-sided testing, the corresponding result holds for locally adaptive transformations of signed residuals, under both point-null and composite-null formulations and their respective vector losses. These results extend the original MRD admissibility theory while identifying the monotone ordering structure that supports it.


The cutoff result is broader than the conventional step-down setting. The
proof uses only a cutoff common to the active hypotheses at the current stage,
not an ordering across stages. Thus arbitrary positive stage-specific
cutoffs are allowed; when they increase, the resulting sequential rule need
not be step-down or step-up. The bivariate example after
Theorem~\ref{THM_ADMISSIBILITY_GENERAL} makes this distinction concrete.

Section~\ref{SECTION_BAYES_LIMITS} gives a separate constructive result
under additive total loss. It uses the same residual-locality property as
the vector-loss argument, but encodes the full adaptive rejection path in
finite-support priors. The resulting approximation yields pointwise
convergence of total-risk functions. Its relation to complete-class theory
is motivational: suitable Bayes-rule closures can be essentially complete
under additional regularity and topology assumptions
\citep{Wald1947EssentiallyCompleteClass,Berger1985SDT}, but we do not
invoke such a characterization here. Since every rule in the two classes
has the representation, it cannot distinguish which, if any, are
admissible under total loss.

Vector-loss admissibility and FDR control answer different questions.
Admissibility rules out uniform componentwise risk improvement, whereas FDR
control bounds the expected proportion of false rejections under specified
assumptions. Neither property implies the other, so the present results do
not provide an FDR guarantee for MRD-type procedures.

The proof's main difficulty is the adaptive path: locally varying scores
can change the selected coordinate, the next conditioning set, and the score
used at later stages. The argument therefore cannot match each generalized
path to the MRD path. Instead, the residual-shift identity, invariance of the
rejection history while the target remains active, and the stopping-stage
principle jointly establish the required interval acceptance geometry.
Along the Matthes--Truax direction, the target residual changes monotonically
while competing active residuals remain fixed. This explains why admissibility
persists even when the generalized rule's rankings and rejection path differ
from MRD's.

For the procedure classes studied here, admissibility under additive total misclassification loss remains open for both testing formulations. Vector-loss inadmissibility would imply
total-loss inadmissibility, since componentwise risk domination is preserved
by summation; the converse need not hold. Under total loss, a risk increase
for one hypothesis may be offset by a reduction for another. Moreover, along
a natural-coordinate line, the competing residuals remain invariant only
while the target hypothesis is active; after its rejection, later residuals
and decisions can change. The individual acceptance-section criterion
therefore does not determine the aggregate risk difference. A proof must
control these cross-hypothesis trade-offs and the subsequent adaptive path
uniformly over the parameter space.

A broader open problem is to characterize complete classes for
multiple-testing procedures under dependence, particularly under
additive total loss. Recent work by \citep{SunWang2026EvaluesCompleteClass}
provides a related result for a different decision problem: among
simultaneous FDR-controlling procedures based on e-values, a class of
weighted-mean closed e-BH procedures forms a complete class. That
result concerns FDR dominance within an e-value procedure class; it
does not characterize procedures under the Gaussian vector or total
losses considered here. It nevertheless illustrates the value of
identifying a suitable procedure class and dominance criterion, and
motivates investigating whether a corresponding complete-class theory
can be developed for covariance-adaptive residual procedures.

The bivariate example settles only the role of cutoff ordering. Whether
strict monotonicity of the residual scores is also necessary for admissibility,
or merely a broad sufficient condition, remains open.

A further important direction concerns covariance uncertainty. The
results in this paper assume that the positive-definite covariance matrix
is known. In practice, it may need to be estimated from an independent
training sample or from the observations used for testing. Estimation
error can perturb the conditional residuals and their standardizing
variances at each stage. Because the residual system is recomputed along
an adaptive rejection path, these perturbations may change the rankings
of active hypotheses and the subsequent conditioning sets. Extending
the present admissibility results to estimated covariance structures
will therefore require a decision-theoretic formulation that accounts
for the covariance estimator as well as an analysis of how its error
affects the residual geometry and rejection path. These questions
connect the present work with high-dimensional covariance and
precision-matrix estimation, including regularization and banding
\citep{LedoitWolf2004, BickelLevina2008}, adaptive thresholding for
sparse covariance matrices \citep{CaiLiu2011}, and sparse precision
estimation \citep{FriedmanHastieTibshirani2008}.

The analysis relies on Gaussian conditional residuals and the Matthes--Truax
criterion. Extending the geometric argument to other exponential families or
dependence models is a natural next step. Another open question is whether
residual-based procedures can combine admissibility with useful FDR control
under general dependence; neither guarantee follows from the other, and a
joint result would require new arguments.

\appendix
\section{Appendix}\label{SECTION_APPENDIX}
\label{app}

This appendix contains all proofs of the results in the paper. It begins
with the proof of the residual-shift identity in
Lemma~\ref{LEM_SHIFT_PROPERTY}, which is used in both admissibility
arguments. It then establishes the index-invariance and stopping-stage
results and the acceptance-interval property for the two-sided procedure,
before proving the one-sided admissibility theorem.

\begin{flushleft}    
	\textbf{Proof of Lemma \ref{LEM_SHIFT_PROPERTY}}
\end{flushleft}

\begin{proof}
	Fix $t$ and an ordered tuple
	$(j_1,\ldots,j_{t-1})$ of distinct indices in $\{2,\ldots,n\}$, and let
	\[
	\mathcal{I}_t=\{1,\ldots,n\}\setminus\{j_1,\ldots,j_{t-1}\}
	\]
	be the active coordinate set. For each $j\in\mathcal{I}_t$, put
	$\mathcal{B}_j=\mathcal{I}_t\setminus\{j\}$. By
	\eqref{MRD_STATISTICS}, the numerator of
	$U_{tj}^{(j_1,\ldots,j_{t-1})}(\boldsymbol{x})$ is
	\[
	N_{tj}(\boldsymbol{x})
	=
	x_j-
	\boldsymbol{\Sigma}_{j\mathcal{B}_j}
	\boldsymbol{\Sigma}_{\mathcal{B}_j\mathcal{B}_j}^{-1}
	\boldsymbol{x}_{\mathcal{B}_j},
	\]
	and its denominator is
	$\sigma_{j\cdot(j_1,\ldots,j_{t-1})}^{1/2}$, where
	\[
	\sigma_{j\cdot(j_1,\ldots,j_{t-1})}
	=
	\sigma_{jj}-
	\boldsymbol{\Sigma}_{j\mathcal{B}_j}
	\boldsymbol{\Sigma}_{\mathcal{B}_j\mathcal{B}_j}^{-1}
	\boldsymbol{\Sigma}_{\mathcal{B}_j j}.
	\]
	When $\mathcal{B}_j$ is empty, the regression terms in these displays
	are interpreted as zero.
	
	Since $\boldsymbol{g}$ is the first column of
	$\boldsymbol{\Sigma}$, perturbing $\boldsymbol{x}$ to
	$\boldsymbol{x}+r\boldsymbol{g}$ changes the numerator by
	\[
	N_{tj}(\boldsymbol{x}+r\boldsymbol{g})-N_{tj}(\boldsymbol{x})
	=
	r\left(
	\sigma_{j1}-
	\boldsymbol{\Sigma}_{j\mathcal{B}_j}
	\boldsymbol{\Sigma}_{\mathcal{B}_j\mathcal{B}_j}^{-1}
	\boldsymbol{\Sigma}_{\mathcal{B}_j1}
	\right).
	\]
	For $j=1$, this coefficient is the Schur complement
	$\sigma_{1\cdot(j_1,\ldots,j_{t-1})}$. Dividing by the fixed
	denominator gives
	\[
	U_{t1}^{(j_1,\ldots,j_{t-1})}(\boldsymbol{x}+r\boldsymbol{g})
	=
	U_{t1}^{(j_1,\ldots,j_{t-1})}(\boldsymbol{x})
	+r\sigma_{1\cdot(j_1,\ldots,j_{t-1})}^{1/2}.
	\]
	
	Now let $j\in\mathcal{I}_t\setminus\{1\}$. Since the history excludes
	coordinate $1$, we have $1\in\mathcal{B}_j$. The column of
	$\boldsymbol{\Sigma}_{\mathcal{B}_j\mathcal{B}_j}$ indexed by coordinate
	$1$ is $\boldsymbol{\Sigma}_{\mathcal{B}_j1}$; hence
	\[
	\boldsymbol{\Sigma}_{\mathcal{B}_j\mathcal{B}_j}^{-1}
	\boldsymbol{\Sigma}_{\mathcal{B}_j1}
	\]
	is the coordinate vector selecting coordinate $1$ within
	$\mathcal{B}_j$. It follows that
	\[
	\boldsymbol{\Sigma}_{j\mathcal{B}_j}
	\boldsymbol{\Sigma}_{\mathcal{B}_j\mathcal{B}_j}^{-1}
	\boldsymbol{\Sigma}_{\mathcal{B}_j1}
	=\sigma_{j1}.
	\]
	Thus the numerator, and therefore the standardized residual, is
	unchanged by the perturbation. This proves both assertions.
\end{proof}

\begin{flushleft}    
	\textbf{Proof of Lemma \ref{LEM_INDEX_INVARIANCE_PARTIAL}}
\end{flushleft}

\begin{proof}
We shall prove the result using the method of induction.\vspace{2mm}

First, we establish the result when $l=1$.\vspace{2mm}

Observe that since both $t>1$ and $t_0>1$, one must have
$$
j_1(\boldsymbol{x}^{*}+r_0\boldsymbol{g})\neq 1
\quad \text{and} \quad
j_1(\boldsymbol{x}^{*})\neq 1.
$$

Then, using Corollary \ref{COR_RESIDUAL_INVARIANCE}, we obtain
\begin{eqnarray}\label{LEM_INDEX_INVARIANCE_PARTIAL_EQ1}
j_1(\boldsymbol{x}^{*}+r_{0}\boldsymbol{g})
&=& \argmax_{j \in \{2,\dots,n\}} S_{1j}(\boldsymbol{x}^{*}+r_{0}\boldsymbol{g})\nonumber\\
&=& \argmax_{j \in \{2,\dots,n\}} S_{1j}(\boldsymbol{x}^{*})\nonumber\\
&=& j_1(\boldsymbol{x}^{*}).
\end{eqnarray}

Hence, the result is true for $l=1$.\vspace{2mm}

Now suppose that, for some $2\leq l\leq \min\{t_0,t\}-1$,
\begin{equation}\label{LEM_INDEX_INVARIANCE_PARTIAL_INDUCTION}
j_k(\boldsymbol{x}^{*}+r_0\boldsymbol{g})=j_k(\boldsymbol{x}^{*}),
\ \textrm{ for all } k=1,\ldots,l-1.
\end{equation}

By the induction hypothesis \eqref{LEM_INDEX_INVARIANCE_PARTIAL_INDUCTION}, the active coordinate sets and the corresponding conditioning index sets coincide at $\boldsymbol{x}^{*}+r_0\boldsymbol{g}$ and $\boldsymbol{x}^{*}$ up to stage $l-1$, which, for $l=2$, follows directly from \eqref{LEM_INDEX_INVARIANCE_PARTIAL_EQ1}.

Our aim is now to show that
$$ j_l(\boldsymbol{x}^{*}+r_0\boldsymbol{g}) = j_l(\boldsymbol{x}^{*}), \textrm{ for all } l=2,\dots, \min\{t_0,t\}-1.$$

Now using Lemma \ref{LEM_SHIFT_PROPERTY}, it follows that for all $ j\in\{2,\dots,n\}\setminus
\{j_1(\boldsymbol{x}^{*}+r_0\boldsymbol{g}),\dots,
	j_{l-1}(\boldsymbol{x}^{*}+r_0\boldsymbol{g})\},$
\begin{equation}
U_{lj}^{(j_1(\boldsymbol{x}^{*}+r_{0}\boldsymbol{g}),\dots,j_{l-1}(\boldsymbol{x}^{*}+r_{0}\boldsymbol{g}))}(\boldsymbol{x}^{*}+r_{0}\boldsymbol{g})=U_{lj}^{(j_1(\boldsymbol{x}^{*}+r_{0}\boldsymbol{g}),\dots,j_{l-1}(\boldsymbol{x}^{*}+r_{0}\boldsymbol{g}))}(\boldsymbol{x}^{*}).\notag
\end{equation}
	
Therefore, we obtain for all $j\in\{2,\dots,n\} \setminus \{j_1(\boldsymbol{x}^{*}+r_{0}\boldsymbol{g}),\dots,j_{l-1}(\boldsymbol{x}^{*}+r_{0}\boldsymbol{g})\},$
\begin{equation}\label{LEM_INDEX_INVARIANCE_PARTIAL_EQ3}
S_{lj}^{(j_1(\boldsymbol{x}^{*}+r_{0}\boldsymbol{g}),\dots,j_{l-1}(\boldsymbol{x}^{*}+r_{0}\boldsymbol{g}))}(\boldsymbol{x}^{*}+r_{0}\boldsymbol{g})=S_{lj}^{(j_1(\boldsymbol{x}^{*}+r_{0}\boldsymbol{g}),\dots,j_{l-1}(\boldsymbol{x}^{*}+r_{0}\boldsymbol{g}))}(\boldsymbol{x}^{*}).
\end{equation}

Again using Lemma \ref{LEM_SHIFT_PROPERTY} coupled with the induction hypothesis \eqref{LEM_INDEX_INVARIANCE_PARTIAL_INDUCTION}, we have
\begin{eqnarray}
U^{(j_1(\boldsymbol{x}^{*}+r_{0}\boldsymbol{g}),\dots,j_{l-1}(\boldsymbol{x}^{*}+r_{0}\boldsymbol{g}))}_{l1}(\boldsymbol{x}^{*} + r\boldsymbol{g})
	&=& U^{(j_1(\boldsymbol{x}^{*}+r_{0}\boldsymbol{g}),\dots,j_{l-1}(\boldsymbol{x}^{*}+r_{0}\boldsymbol{g}))}_{l1}(\boldsymbol{x}^{*})\notag\\
	&& \quad +
	r\sigma^{\frac{1}{2}}_{1\cdot( j_1(\boldsymbol{x}^{*}+r_{0}\boldsymbol{g}),\dots,j_{l-1}(\boldsymbol{x}^{*}+r_{0}\boldsymbol{g}))}\nonumber\\
	&=& U^{(j_1(\boldsymbol{x}^{*}),\dots,j_{l-1}(\boldsymbol{x}^{*}))}_{l1}(\boldsymbol{x}^{*}) +
	r\sigma^{\frac{1}{2}}_{1\cdot( j_1(\boldsymbol{x}^{*}),\dots,j_{l-1}(\boldsymbol{x}^{*}))},\nonumber
\end{eqnarray}
which shows that, for each $1<l<\min\{t_0,t\}$, the perturbation affects only the statistic corresponding to coordinate $1$.

Since $H_{01}$ is rejected only at stage $t_0$ when $\boldsymbol{x}^{*}+r_0\boldsymbol{g}$ is observed, for every $1 <l<\min\{t_0,t\}$, the statistic
$$
S_{l1}^{(j_1(\boldsymbol{x}^{*}+r_0\boldsymbol{g}),\dots,
	j_{l-1}(\boldsymbol{x}^{*}+r_0\boldsymbol{g}))}
(\boldsymbol{x}^{*}+r_0\boldsymbol{g})
$$
cannot attain the maximum among the active statistics at stage $l$.
Consequently, the maximizing index $j_{l}(\boldsymbol{x}^{*}+r_0\boldsymbol{g})$ at stage $l<t_0$ when $\boldsymbol{x}^{*}+r_0\boldsymbol{g}$ is observed must belong to the competing set
$$
\{2,\dots,n\}\setminus
\{j_1(\boldsymbol{x}^{*}+r_0\boldsymbol{g}),\dots,
	j_{l-1}(\boldsymbol{x}^{*}+r_0\boldsymbol{g})\},
$$
whose associated statistics remain invariant under the perturbation.

Hence, using \eqref{LEM_INDEX_INVARIANCE_PARTIAL_EQ1},
\eqref{LEM_INDEX_INVARIANCE_PARTIAL_EQ3}, and the induction hypothesis
\eqref{LEM_INDEX_INVARIANCE_PARTIAL_INDUCTION}, we obtain
\begin{eqnarray}
j_l(\boldsymbol{x}^{*}+r_{0}\boldsymbol{g})
&=& \argmax_{j \in \{2,\dots,n\}\setminus\{j_1(\boldsymbol{x}^{*}+r_{0}\boldsymbol{g}), \dots, j_{l-1}(\boldsymbol{x}^{*}+r_{0}\boldsymbol{g})\}} S_{lj}^{(j_1(\boldsymbol{x}^{*}+r_{0}\boldsymbol{g}),\dots, j_{l-1}(\boldsymbol{x}^{*}+r_{0}\boldsymbol{g}))}(\boldsymbol{x}^{*}+r_{0}\boldsymbol{g})\nonumber\\
&=& \argmax_{j \in \{2,\dots,n\}\setminus\{j_1(\boldsymbol{x}^{*}+r_{0}\boldsymbol{g}),\dots, j_{l-1}(\boldsymbol{x}^{*}+r_{0}\boldsymbol{g})\}} S_{lj}^{(j_1(\boldsymbol{x}^{*}+r_{0}\boldsymbol{g}),\dots, j_{l-1}(\boldsymbol{x}^{*}+r_{0}\boldsymbol{g}))}(\boldsymbol{x}^{*})\nonumber\\
&=& \argmax_{j \in \{2,\dots,n\}\setminus\{j_1(\boldsymbol{x}^{*}), \dots, j_{l-1}(\boldsymbol{x}^{*})\}} S_{lj}^{(j_1(\boldsymbol{x}^{*}), \dots, j_{l-1}(\boldsymbol{x}^{*}))}(\boldsymbol{x}^{*})\nonumber\\
&=& j_l(\boldsymbol{x}^{*}),\nonumber
\end{eqnarray}
for all $l=2,\dots,\min\{t_0,t\}-1$. This completes the proof of Lemma \ref{LEM_INDEX_INVARIANCE_PARTIAL}.	
\end{proof}

\begin{flushleft}
	\textbf{Proof of Lemma \ref{LEM_STOPPING_STAGE}}
\end{flushleft}
\begin{proof}
	
First observe that, when $t_0=1$, the conclusion $t_0\le t$ holds trivially, since $t\ge 1$. On the other hand, if $t=1$, then the procedure stops at the first stage when $\boldsymbol{x}^{*}$ is observed, so $S_{1j}(\boldsymbol{x}^{*})\le C_1^G$ for all $j=1,\dots,n$. In that case, if $t_0>1$, then at stage 1, some coordinate $k\ne 1$ must be rejected when
$\boldsymbol{x}^{*}+r_0\boldsymbol{g}$ is observed, so
$S_{1k}(\boldsymbol{x}^{*}+r_0\boldsymbol{g})>C_1^G$. But by Corollary~\ref{COR_RESIDUAL_INVARIANCE},
$S_{1k}(\boldsymbol{x}^{*}+r_0\boldsymbol{g})=S_{1k}(\boldsymbol{x}^{*})$, a contradiction. Hence $t_0=1$, and
therefore $t_0\le t$.

Thus it remains to consider the case $t>1$ and $t_0>1$.
	
Since $t>1$, the process does not stop in the first $(t-1)$ steps without rejecting $H_{01}$ when $\boldsymbol{x}^{*}$ is observed. Consequently, we must have
\begin{eqnarray}
S_{lj_l(\boldsymbol{x}^{*})}^{(j_1(\boldsymbol{x}^{*}),\dots,j_{l-1}(\boldsymbol{x}^{*}))}(\boldsymbol{x}^{*}) > C^{G}_{l}\nonumber
\end{eqnarray}
for all $l=1,\dots,t-1$, where
$j_l(\boldsymbol{x}^{*}) \neq 1$ for each $l$, and
	\begin{eqnarray}
		S_{tj_{t}(\boldsymbol{x}^{*})}^{(j_1(\boldsymbol{x}^{*}),\dots,j_{t-1}(\boldsymbol{x}^{*}))}(\boldsymbol{x}^{*})
		\leqslant
		C^{G}_{t},
		\nonumber
	\end{eqnarray}
	which implies that
	\begin{eqnarray}
		S_{tj}^{(j_1(\boldsymbol{x}^{*}),\dots,j_{t-1}(\boldsymbol{x}^{*}))}(\boldsymbol{x}^{*})
		\leqslant
		C^{G}_{t}
		\nonumber
	\end{eqnarray}
	for all
	$j \in \{1,\dots,n\}\setminus\{j_1(\boldsymbol{x}^{*}),\dots,j_{t-1}(\boldsymbol{x}^{*})\}$.
	Since
	$j_l(\boldsymbol{x}^{*}) \neq 1$
	for all $l=1,\dots,t-1$, the index $1$ remains active at stage $t$.
	
	
	On the contrary, let us now assume that $t_0>t$. Then
	\begin{eqnarray}\label{LEM_STOPPING_STAGE_EQ1}
		S_{tj_t(\boldsymbol{x}^{*}+r_{0}\boldsymbol{g})}^{(j_1(\boldsymbol{x}^{*}+r_{0}\boldsymbol{g}),\dots,j_{t-1}(\boldsymbol{x}^{*}+r_{0}\boldsymbol{g}))}(\boldsymbol{x}^{*}+r_{0}\boldsymbol{g})>	C^{G}_{t}, 
	\end{eqnarray}
	otherwise the process would have stopped at stage $t$ without rejecting $H_{01}$ when $\boldsymbol{x}^{*}+r_{0}\boldsymbol{g}$ is observed, which would be a contradiction.
	
%
	
	
Let
\begin{equation}
	k=j_t(\boldsymbol{x}^{*}+r_0\boldsymbol{g}).\notag
\end{equation}

Since we are assuming that $t_0>t$, the hypothesis $H_{01}$ is not rejected at stage $t$ when $\boldsymbol{x}^{*}+r_0\boldsymbol{g}$ is observed. Hence $k\neq 1$.
Moreover, by Corollary \ref{COR_INDEX_INVARIANCE},
\begin{equation}\label{LEM_STOPPING_STAGE_EQ2}
	j_l(\boldsymbol{x}^{*}+r_0\boldsymbol{g})=j_l(\boldsymbol{x}^{*}),
	\ \textrm{ for } l=1,\dots,t-1.
\end{equation}

Therefore $k$ remains active at stage $t$ when $\boldsymbol{x}^{*}$ is observed. That is, $H_{0k}$ is not rejected, at least, until stage $t$ when $\boldsymbol{x}^{*}$ is observed. Since $k\neq 1$, using Lemma \ref{LEM_SHIFT_PROPERTY} we obtain
\begin{eqnarray}\label{LEM_STOPPING_STAGE_EQ3}
U_{tk}^{(j_1(\boldsymbol{x}^{*}+r_0\boldsymbol{g}),\dots,
		j_{t-1}(\boldsymbol{x}^{*}+r_0\boldsymbol{g}))}
	(\boldsymbol{x}^{*}+r_0\boldsymbol{g})
	&=&
	U_{tk}^{(j_1(\boldsymbol{x}^{*}+r_0\boldsymbol{g}),\dots,
		j_{t-1}(\boldsymbol{x}^{*}+r_0\boldsymbol{g}))}
	(\boldsymbol{x}^{*})
	\nonumber\\
	&=&
	U_{tk}^{(j_1(\boldsymbol{x}^{*}),\dots,
		j_{t-1}(\boldsymbol{x}^{*}))}
	(\boldsymbol{x}^{*}).
\end{eqnarray}

It is important to observe that, due to \eqref{LEM_STOPPING_STAGE_EQ2}, the locally adaptive transformation functions corresponding to the generalized residual statistics in \eqref{LEM_STOPPING_STAGE_EQ3} also coincide at
$\boldsymbol{x}^{*}+r_0\boldsymbol{g}$ and $\boldsymbol{x}^{*}$. Therefore, since each generalized residual score is
obtained through the same strictly increasing transformation applied to the
absolute value of the corresponding residual statistic, \eqref{LEM_STOPPING_STAGE_EQ3} coupled with \eqref{LEM_STOPPING_STAGE_EQ1}, implies
\begin{eqnarray}\label{LEM_STOPPING_STAGE_EQ4}
S_{tk}^{(j_1(\boldsymbol{x}^{*}),\dots,j_{t-1}(\boldsymbol{x}^{*}))}
	(\boldsymbol{x}^{*})
	&=&
	S_{tk}^{(j_1(\boldsymbol{x}^{*}+r_0\boldsymbol{g}),\dots,
		j_{t-1}(\boldsymbol{x}^{*}+r_0\boldsymbol{g}))}
	(\boldsymbol{x}^{*}+r_0\boldsymbol{g})
	\nonumber\\
	&>& C_t^G.
\end{eqnarray}


But $k$ is active at stage $t$ when $\boldsymbol{x}^{*}$ is observed. Since the procedure stops at stage $t$ under $\boldsymbol{x}^{*}$, we must have
\begin{equation}
	S_{tk}^{(j_1(\boldsymbol{x}^{*}),\dots,j_{t-1}(\boldsymbol{x}^{*}))}(\boldsymbol{x}^{*})
	\le C_t^G,\nonumber
\end{equation}
which contradicts \eqref{LEM_STOPPING_STAGE_EQ4}. This completes the proof of Lemma~\ref{LEM_STOPPING_STAGE}.
\end{proof}

\begin{flushleft}
	\textbf{Proof of Lemma \ref{LEM_INTERVAL_PROPERTY}}
\end{flushleft}

\begin{proof}
Since $\phi_1(\boldsymbol{x}^{*})=0$, the point $\boldsymbol{x}^{*}$ belongs to the acceptance region of $H_{01}$. Hence, the testing procedure must stop before rejecting $H_{01}$. Let the testing procedure stop at some stage $t$ without rejecting $H_{01}$ when $\boldsymbol{x}^{*}$ is observed.

Again, since $\phi_1(\boldsymbol{x}^{*}+r_0\boldsymbol{g})=1$, the point $\boldsymbol{x}^{*}+r_0\boldsymbol{g}$ belongs to the rejection region of $H_{01}$. Let the testing procedure reject $H_{01}$ at some stage $t_0$ when $\boldsymbol{x}^{*}+r_0\boldsymbol{g}$ is observed.
	
Let us first consider the situation when $t_0>1$.
	
By Lemma~\ref{LEM_STOPPING_STAGE}, we have $t_0\le t$, and hence the
	indices $j_1(\boldsymbol{x}^{*}),\ldots,j_{t_0}(\boldsymbol{x}^{*})$ are well-defined.
	
Observe first that if
\begin{equation}
	j_{t_0}(\boldsymbol{x}^{*})=1,\nonumber
\end{equation}
the test fails to reject $H_{01}$ at stage $t_{0}$ when $\boldsymbol{x}^{*}$ is observed. Consequently,
\begin{eqnarray}
S_{t_{0}1}^{(j_1(\boldsymbol{x}^{*}),\dots,j_{t_{0}-1}(\boldsymbol{x}^{*}))}(\boldsymbol{x}^{*})
&=&	S_{t_{0}j_{t_0}(\boldsymbol{x}^{*})}^{(j_1(\boldsymbol{x}^{*}),\dots,j_{t_{0}-1}(\boldsymbol{x}^{*}))}(\boldsymbol{x}^{*})\nonumber\\
&\leqslant& C_{t_0}^G \nonumber\\
&<& S_{t_{0}j_{t_0}(\boldsymbol{x}^{*}+r_{0}\boldsymbol{g})}^{(j_1(\boldsymbol{x}^{*}+r_{0}\boldsymbol{g}),\dots,
	j_{t_{0}-1}(\boldsymbol{x}^{*}+r_{0}\boldsymbol{g}))}
(\boldsymbol{x}^{*}+r_{0}\boldsymbol{g})\nonumber\\
&=& S_{t_{0}1}^{(j_1(\boldsymbol{x}^{*}+r_{0}\boldsymbol{g}),\dots,
		j_{t_{0}-1}(\boldsymbol{x}^{*}+r_{0}\boldsymbol{g}))}
	(\boldsymbol{x}^{*}+r_{0}\boldsymbol{g})\nonumber\\
&=& S_{t_{0}1}^{(j_1(\boldsymbol{x}^{*}),\dots,
		j_{t_{0}-1}(\boldsymbol{x}^{*}))}
	(\boldsymbol{x}^{*}+r_{0}\boldsymbol{g}),
	\nonumber
\end{eqnarray}
where the strict inequality in the above chain of inequalities follows from the fact that $H_{01}$ is rejected at stage $t_0$ when $\boldsymbol{x}^{*}+r_0\boldsymbol{g}$ is observed.

Otherwise, if
\begin{equation}
	j_{t_0}(\boldsymbol{x}^{*})\neq 1,\nonumber
\end{equation}
then using Corollary~\ref{COR_RESIDUAL_INVARIANCE} together with Corollary \ref{COR_INDEX_INVARIANCE} we have
\begin{eqnarray}
S_{t_{0}1}^{(j_1(\boldsymbol{x}^{*}),\dots,j_{t_{0}-1}(\boldsymbol{x}^{*}))}(\boldsymbol{x}^{*})	
&<& S_{t_{0}j_{t_{0}}(\boldsymbol{x}^{*})}^{(j_1(\boldsymbol{x}^{*}),\dots,
		j_{t_{0}-1}(\boldsymbol{x}^{*}))}(\boldsymbol{x}^{*})\nonumber\\
&=& S_{t_{0}j_{t_{0}}(\boldsymbol{x}^{*})}^{(j_1(\boldsymbol{x}^{*}),\dots,
		j_{t_{0}-1}(\boldsymbol{x}^{*}))}(\boldsymbol{x}^{*}+r_{0}\boldsymbol{g})
	\nonumber\\
&=& S_{t_{0}j_{t_{0}}(\boldsymbol{x}^{*})}^{(j_1(\boldsymbol{x}^{*}+r_{0}\boldsymbol{g}),\dots,
		j_{t_{0}-1}(\boldsymbol{x}^{*}+r_{0}\boldsymbol{g}))}
	(\boldsymbol{x}^{*}+r_{0}\boldsymbol{g})
	\nonumber\\
&\leqslant&
	S_{t_{0}1}^{(j_1(\boldsymbol{x}^{*}+r_{0}\boldsymbol{g}),\dots,
		j_{t_{0}-1}(\boldsymbol{x}^{*}+r_{0}\boldsymbol{g}))}
	(\boldsymbol{x}^{*}+r_{0}\boldsymbol{g})
	\nonumber\\
&=&
	S_{t_{0}1}^{(j_1(\boldsymbol{x}^{*}),\dots,
		j_{t_{0}-1}(\boldsymbol{x}^{*}))}
	(\boldsymbol{x}^{*}+r_{0}\boldsymbol{g}).
	\nonumber
\end{eqnarray}	
	
	
Thus, in either case, it is shown that
\begin{eqnarray}
	S_{t_{0}1}^{(j_1(\boldsymbol{x}^{*}),\dots,j_{t_{0}-1}(\boldsymbol{x}^{*}))}(\boldsymbol{x}^{*})	
	&<& 	S_{t_{0}1}^{(j_1(\boldsymbol{x}^{*}),\dots,
		j_{t_{0}-1}(\boldsymbol{x}^{*}))}
	(\boldsymbol{x}^{*}+r_{0}\boldsymbol{g}).
	\nonumber
\end{eqnarray}	

By Corollary \ref{COR_INDEX_INVARIANCE}, the adaptive histories up to stage $(t_0-1)$ coincide at
$\boldsymbol{x}^{*}$ and $\boldsymbol{x}^{*}+r_0\boldsymbol{g}$. Hence the locally adaptive transformation function
associated with the statistic
\begin{equation}
S^{(j_1(\boldsymbol{x}^{*}),\ldots,j_{t_0-1}(\boldsymbol{x}^{*}))}_{t_0 1}\notag
\end{equation}
is the same at the two points $\boldsymbol{x}^{*}$ and $\boldsymbol{x}^{*}+r_0\boldsymbol{g}$. Therefore, the preceding
comparison of the corresponding $S$-statistics may be equivalently translated
into a comparison of the absolute residual statistics.	
	
Since for fixed $j_1,\dots,j_{t_{0}-1},$ $S_{t_{0}1}^{(j_1,\dots,j_{t_{0}-1})}$ is a strictly increasing function of
	$|U_{t_{0}1}^{(j_1,\dots,j_{t_{0}-1})}|,$ it follows that
	\begin{align}
		|U_{t_{0}1}^{(j_1(\boldsymbol{x}^{*}),\dots,j_{t_{0}-1}(\boldsymbol{x}^{*}))}(\boldsymbol{x}^{*})|
		&< |U_{t_{0}1}^{(j_1(\boldsymbol{x}^{*}),\dots,j_{t_{0}-1}(\boldsymbol{x}^{*}))}(\boldsymbol{x}^{*}+r_{0}\boldsymbol{g})|.\label{LEM_SHIFT_PROPERTY_EQ1}
	\end{align}

We now show that
\begin{equation}
	U^{(j_1(\boldsymbol{x}^{*}),\ldots,j_{t_0-1}(\boldsymbol{x}^{*}))}_{t_01}
	(\boldsymbol{x}^{*}+r_0\boldsymbol{g})>0.\notag
\end{equation}
On the contrary, suppose the above claim is not true. Since
\begin{equation}
U^{(j_1(\boldsymbol{x}^{*}),\ldots,j_{t_0-1}(\boldsymbol{x}^{*}))}_{t_01}
(\boldsymbol{x}^{*}+r\boldsymbol{g})\notag
\end{equation}
is strictly increasing in $r$ and $r_0>0$, we must have
\begin{equation}
	U^{(j_1(\boldsymbol{x}^{*}),\ldots,j_{t_0-1}(\boldsymbol{x}^{*}))}_{t_01}
	(\boldsymbol{x}^{*})
	<
	U^{(j_1(\boldsymbol{x}^{*}),\ldots,j_{t_0-1}(\boldsymbol{x}^{*}))}_{t_01}
	(\boldsymbol{x}^{*}+r_0\boldsymbol{g})
	\le 0.\notag
\end{equation}
The preceding step coupled with Remark
\ref{REMARK_MONOTONE_STRUCTURE} would imply
\begin{equation}
|U^{(j_1(\boldsymbol{x}^{*}),\ldots,j_{t_0-1}(\boldsymbol{x}^{*}))}_{t_01}
	(\boldsymbol{x}^{*})|
	>
|U^{(j_1(\boldsymbol{x}^{*}),\ldots,j_{t_0-1}(\boldsymbol{x}^{*}))}_{t_01}
	(\boldsymbol{x}^{*}+r_0\boldsymbol{g})|,\notag
\end{equation}
contradicting \eqref{LEM_SHIFT_PROPERTY_EQ1}.


Thus we obtain
\begin{equation}
	U^{(j_1(\boldsymbol{x}^{*}),\dots,j_{t_0-1}(\boldsymbol{x}^{*}))}_{t_01}
	(\boldsymbol{x}^{*}+r_0\boldsymbol{g})>0.\notag
\end{equation}
	
	
	But for given $(j_1,\dots,j_{t_{0}-1})$, the function $U_{t_{0}1}^{(j_1,\dots,j_{t_{0}-1})}(\boldsymbol{x}^{*}+r\boldsymbol{g})$ is strictly increasing in $r$. Hence for all $r > r_0,$ we have
	\begin{eqnarray}\label{LEM_SHIFT_PROPERTY_EQ2}
		U_{t_{0}1}^{(j_1(\boldsymbol{x}^{*}),\dots,j_{t_{0}-1}(\boldsymbol{x}^{*}))}(\boldsymbol{x}^{*}+r\boldsymbol{g})
		&>& U_{t_{0}1}^{(j_1(\boldsymbol{x}^{*}),\dots,j_{t_{0}-1}(\boldsymbol{x}^{*}))}(\boldsymbol{x}^{*}+r_{0}\boldsymbol{g})>0.
	\end{eqnarray}
	
	We now complete the proof now based on a contrapositive argument. Recall that, we need to show $\phi_1(\boldsymbol{x}^{*}+r\boldsymbol{g})=1$ for all $r > r_0.$ On the contrary, suppose this is not true. Then there exists some $r_1 > r_0 $ such that $\phi_1(\boldsymbol{x}^{*}+r_1\boldsymbol{g})=0.$ Let $t_1$ denote the step at which the testing procedure must stop without rejecting $H_{01}$ when $\boldsymbol{x}^{*}+r_1\boldsymbol{g}$ is observed.
	

Let us now view $\boldsymbol{x}^{*}+r_{1}\boldsymbol{g}$ as the new acceptance point and observe that
\begin{equation}\notag
	\boldsymbol{x}^{*}+r_{0}\boldsymbol{g}
	=
	(\boldsymbol{x}^{*}+r_{1}\boldsymbol{g})
	-
	(r_{1}-r_{0})\boldsymbol{g},
\end{equation}
where $r_{1}-r_{0}>0$. Thus, relative to the new base point
$\boldsymbol{x}^{*}+r_{1}\boldsymbol{g}$, the point
$\boldsymbol{x}^{*}+r_{0}\boldsymbol{g}$ becomes a rejection point obtained through a nonzero perturbation along the admissibility direction $\boldsymbol{g}$ with perturbation
parameter $-(r_{1}-r_{0})$. Since Lemma \ref{LEM_INDEX_INVARIANCE_PARTIAL} and Lemma \ref{LEM_STOPPING_STAGE} together with Corollary \ref{COR_INDEX_INVARIANCE} remain valid for arbitrary nonzero perturbations, the preceding
arguments continue to hold under this reversed perturbation as well. Consequently
%
\begin{equation}
	t_0 \le t_1.
	\notag
\end{equation}
Similarly, applying Corollary \ref{COR_INDEX_INVARIANCE} with the same base point $\boldsymbol{x}^{*}+r_{1}\boldsymbol{g}$ and perturbation $-(r_{1}-r_{0})$ together with the assumption $t_{0}> 1$, we obtain
\begin{equation}
	j_l(\boldsymbol{x}^{*} + r_1 \boldsymbol{g})
	=
	j_l(\boldsymbol{x}^{*} + r_0 \boldsymbol{g})
	=
	j_l(\boldsymbol{x}^{*})
	\notag
\end{equation}
for all $l = 1,\ldots,t_0-1$.
	
Applying the preceding arguments with the new base point $\boldsymbol{x}^{*}+r_1\boldsymbol{g}$ and the reversed perturbation $-(r_{1}-r_{0})$, we obtain 
	\begin{eqnarray}
		|U_{t_{0}1}^{(j_1(\boldsymbol{x}^{*}),\dots,j_{t_{0}-1}(\boldsymbol{x}^{*}))}(\boldsymbol{x}^{*}+r_1\boldsymbol{g})|
		&<& |U_{t_{0}1}^{(j_1(\boldsymbol{x}^{*}),\dots,j_{t_{0}-1}(\boldsymbol{x}^{*}))}(\boldsymbol{x}^{*}+r_{0}\boldsymbol{g})|\nonumber
	\end{eqnarray}
	which contradicts (\ref{LEM_SHIFT_PROPERTY_EQ2}). Therefore one must have $\phi_1(\boldsymbol{x}^{*}+r\boldsymbol{g})=1 $ for all $r>r_0$, when $t_0>1$.
	
Next consider the case when $t_0=1$. Then, by definition,
\begin{equation}
j_1(\boldsymbol{x}^{*}+r_{0}\boldsymbol{g})=1,\notag
\end{equation}
and hence
\begin{equation}
S_{11}(\boldsymbol{x}^{*}+r_{0}\boldsymbol{g})>C_1^G.\notag
\end{equation}

Since $\phi_1(\boldsymbol{x}^{*})=0$, the procedure does not reject $H_{01}$ when
$\boldsymbol{x}^{*}$ is observed. Therefore,
\begin{equation}
S_{11}(\boldsymbol{x}^{*})\le C_1^G.\notag
\end{equation}

Consequently,
\begin{equation}
S_{11}(\boldsymbol{x}^{*})<S_{11}(\boldsymbol{x}^{*}+r_{0}\boldsymbol{g}).\notag
\end{equation}

Now, by Corollary~\ref{COR_RESIDUAL_INVARIANCE},
\begin{equation}
S_{1j}(\boldsymbol{x}^{*}+r \boldsymbol{g})=S_{1j}(\boldsymbol{x}^{*})
\quad \text{for all } j\in\{2,\dots,n\}
\text{ and all } r\in\mathbb{R}.\nonumber
\end{equation}
Thus, along the admissibility direction $\boldsymbol{g}$, the competing first-stage
statistics remain invariant, while only the statistic corresponding to
coordinate $1$ evolves with $r$.

Since $S_{11}$ is a strictly increasing function of $|U_{11}|$, and
$U_{11}(\boldsymbol{x}^{*}+r\boldsymbol{g})$ is strictly increasing in $r$, Remark \ref{REMARK_MONOTONE_STRUCTURE} implies that
\begin{equation}
U_{11}(\boldsymbol{x}^{*}+r_0\boldsymbol{g})>0.\notag
\end{equation}

Hence, for all $r>r_0$,
\begin{equation}\label{LEM_SHIFT_PROPERTY_EQ4}
U_{11}(\boldsymbol{x}^{*}+r\boldsymbol{g})>U_{11}(\boldsymbol{x}^{*}+r_0\boldsymbol{g})>0.
\end{equation}

Since $t_{0}=1$, \eqref{LEM_SHIFT_PROPERTY_EQ4} above implies for all $r>r_0$,
\begin{equation}\label{LEM_SHIFT_PROPERTY_EQ5}
S_{11}(\boldsymbol{x}^{*}+r\boldsymbol{g})>S_{11}(\boldsymbol{x}^{*}+r_0\boldsymbol{g})> C^{G}_{1}.
\end{equation}

Since $t_0=1$, using the preceding arguments together with Corollary \ref{COR_RESIDUAL_INVARIANCE}, for all $r > r_0$, and for all $j\in\{2,\dots,n\}$, we obtain
\begin{eqnarray}\label{LEM_SHIFT_PROPERTY_EQ6}
	S_{11}(\boldsymbol{x}^{*}+r\boldsymbol{g})
	&>& S_{11}(\boldsymbol{x}^{*}+r_{0}\boldsymbol{g})\nonumber\\
	&\geqslant& S_{1j}(\boldsymbol{x}^{*}+r_{0}\boldsymbol{g})\nonumber\\
	&=& S_{1j}(\boldsymbol{x}^{*})\nonumber\\
	&=& S_{1j}(\boldsymbol{x}^{*}+r\boldsymbol{g}).
\end{eqnarray}

\eqref{LEM_SHIFT_PROPERTY_EQ5} and \eqref{LEM_SHIFT_PROPERTY_EQ6} together imply that every $\boldsymbol{x}^{*}+r\boldsymbol{g}$ will be a point of rejection for $H_{01}$ for all $r > r_0$, that is, $\phi_1(\boldsymbol{x}^{*}+r\boldsymbol{g})=1 $ for all $r > r_0$ when $t_0=1$. This completes the proof of Lemma \ref{LEM_INTERVAL_PROPERTY}.

%
%
%
	
\end{proof}





%

\begin{flushleft}
\textbf{Proof of Theorem~\ref{THM_ONE_SIDED_ADMISSIBILITY}}
\end{flushleft}

The proof uses the residual-shift identity in
Lemma~\ref{LEM_SHIFT_PROPERTY}, but does not require the additional
path-invariance and stopping-stage lemmas developed for the two-sided
case. There, the absolute residual can decrease and then increase along
the admissibility direction, so those lemmas are needed to establish the
interval acceptance section. Here, the signed residual and its strictly
increasing score both increase along that direction. A direct induction
over the rejection path therefore establishes the required half-line
acceptance section.

\begin{proof}
It is enough to establish admissibility of the induced test for testing $H_{01}$; the argument for any other null hypothesis will follow analogously. Put $\boldsymbol{Y}=\boldsymbol{\Sigma}^{-1}\boldsymbol{X}$, let
$\boldsymbol{e}_1=(1,0,\ldots,0)^{\prime}\in\mathbb{R}^n$ be the first
standard basis vector, and set
$\boldsymbol{g}=\boldsymbol{\Sigma}\boldsymbol{e}_1
=(\sigma_{11},\sigma_{21},\ldots,\sigma_{n1})^{\prime}$, the first column
of $\boldsymbol{\Sigma}$. Varying $y_1$ while holding
$y_2,\ldots,y_n$ fixed is equivalent to considering
$\boldsymbol{x}(r)=\boldsymbol{x}+r\boldsymbol{g}$, since
$\boldsymbol{\Sigma}^{-1}\boldsymbol{x}(r)
=\boldsymbol{\Sigma}^{-1}\boldsymbol{x}+r\boldsymbol{e}_1$.

Fix an ordered tuple $(i_1,\ldots,i_{t-1})$ of previously rejected
coordinates that does not contain coordinate $1$. By Lemma~\ref{LEM_SHIFT_PROPERTY}, the MRD shift identity states that
\[
 U_{t1}^{(i_1,\ldots,i_{t-1})}(\boldsymbol{x}(r))
 =U_{t1}^{(i_1,\ldots,i_{t-1})}(\boldsymbol{x})
   +r\sigma_{1\cdot(i_1,\ldots,i_{t-1})}^{1/2},
 \qquad \sigma_{1\cdot(i_1,\ldots,i_{t-1})}>0,
\]
while, for each other active $j$, the residual
$U_{tj}^{(i_1,\ldots,i_{t-1})}(\boldsymbol{x}(r))$ is constant in $r$.
Therefore, at any fixed stage and tuple, the score of coordinate $1$ is
strictly increasing in $r$, whereas each competing score remains constant.

At each fixed history, the residual directions of two distinct active
coordinates are linearly independent by
Lemma~\ref{LEM_BAYES_RESIDUAL_LOCALITY}. Their residuals therefore have a
two-dimensional Lebesgue density. Since the score transformations are
strictly increasing, the set on which two such scores are equal has
two-dimensional Lebesgue measure zero: for each value of one residual,
there is at most one value of the other residual giving a tie. There are
only finitely many histories and coordinate pairs, so the union of all
score-tie sets is Lebesgue-null. Fubini's theorem now allows us to fix,
for almost every $(y_2,\ldots,y_n)$, a line on which no two competing
coordinates tie at any history. At any fixed history, coordinate $1$'s
score is strictly increasing in $r$, whereas each competing score is
constant; hence it can tie each competitor at most once. Only finitely
many values of $r$ need therefore be excluded. Fix such a line
$\{\boldsymbol{x}(r):r\in\mathbb{R}\}$ and a value
$r_0$ at which the procedure rejects $H_{01}$, first at stage $t_0$. Let
$(j_1,\ldots,j_{t_0-1})$ be the ordered tuple rejected before that stage.
At each preceding stage $\ell<t_0$, the history is
$(j_1,\ldots,j_{\ell-1})$, and $j_\ell$ is the unique maximizer among the
active scores at $r_0$. In particular, its score exceeds $C_\ell^G$ and
strictly exceeds the score of coordinate $1$. At this fixed history, all
scores of active coordinates other than $1$ remain constant as $r$ varies.
Thus $j_\ell$ remains the largest among those competing coordinates,
whereas the score of coordinate $1$ increases strictly.

Now take any $r>r_0$, excluding the isolated values at which the score of
coordinate $1$ ties a competing score. We follow the procedure inductively
through the stages before $t_0$. Suppose the same history
$(j_1,\ldots,j_{\ell-1})$ has been reached. If coordinate $1$'s score is
below that of $j_\ell$, then $j_\ell$ is still the maximizer and is rejected,
because its unchanged score exceeds $C_\ell^G$. If coordinate $1$'s score
has overtaken that of $j_\ell$, it is now the maximizer and also exceeds
$C_\ell^G$, since the cutoff is common to the active coordinates. Hence
$H_{01}$ is rejected at stage $\ell$. Repeating this argument, either
$H_{01}$ is rejected before stage $t_0$, or the original history is reached
at stage $t_0$. At that stage coordinate $1$ was the unique maximizer at
$r_0$ and exceeded $C_{t_0}^G$; its score increases while every other
active score remains unchanged. It is therefore still selected and rejected.
Consequently, outside the null set of tie points,
\[
 \phi_1(\boldsymbol{x}(r_0))=1
 \quad\Longrightarrow\quad
 \phi_1(\boldsymbol{x}(r))=1\quad\text{for every }r>r_0.
\]
Thus, for almost every fixed $(y_2,\ldots,y_n)$, the rejection section in
$y_1$ is an upper half-line and the acceptance section is a lower half-line,
up to a null set of tie points. Lemma~\ref{LEM_ONE_SIDED_MATH_TRUAX} then
implies admissibility of the induced test for coordinate $1$ under either
corresponding scalar loss. The same argument applies to each coordinate.
If a competing multiple-testing procedure dominated the stated vector risk,
then each of its coordinate tests would have risk no larger everywhere;
any strict improvement in a component at any parameter point would make
that coordinate test dominate the corresponding induced scalar test. This
contradicts the scalar admissibility just established. Hence the original
procedure is admissible under the corresponding vector loss. This concludes
the proof of Theorem~\ref{THM_ONE_SIDED_ADMISSIBILITY}.
\end{proof}

\begin{flushleft}
	\textbf{Proof of Lemma~\ref{LEM_BAYES_RESIDUAL_LOCALITY}}
\end{flushleft}
\begin{proof}
	Fix a history $r=(i_1,\ldots,i_{t-1})$, and write
	$R=\{i_1,\ldots,i_{t-1}\}$ and
	$\mathcal I_t=\{1,\ldots,n\}\setminus R$. For each active coordinate $j\in\mathcal I_t$, define
	$\mathcal B_j=\mathcal I_t\setminus\{j\}$, the set of other coordinates active at stage $t$. By
	\eqref{MRD_STATISTICS},
	\[
	U_{tj}^{(r)}(\boldsymbol X)
	=
	\frac{N_j(\boldsymbol X)}{\sqrt{\sigma_{j\cdot r}}},
	\qquad
	N_j(\boldsymbol X)
	=
	X_j-\boldsymbol\Sigma_{j\mathcal B_j}
	\boldsymbol\Sigma_{\mathcal B_j\mathcal B_j}^{-1}
	\boldsymbol X_{\mathcal B_j},
	\]
	where
	\[
	\sigma_{j\cdot r}
	=
	\sigma_{jj}-
	\boldsymbol\Sigma_{j\mathcal B_j}
	\boldsymbol\Sigma_{\mathcal B_j\mathcal B_j}^{-1}
	\boldsymbol\Sigma_{\mathcal B_jj}>0.
	\]
	When $\mathcal B_j$ is empty, the regression term is interpreted as
	zero. Since $\boldsymbol X=\boldsymbol\Sigma\boldsymbol Y$, the
	residual is linear in $\boldsymbol Y$. Therefore, we write
	\[
	U_{tj}^{(r)}(\boldsymbol X)
	=\boldsymbol v_{j,R}^{\mathsf T}\boldsymbol Y.
	\]
	
Fix any $k\in\mathcal B_j$. Relabel the coordinates so that $k$ plays
the role of coordinate $1$ in Lemma~\ref{LEM_SHIFT_PROPERTY}. Since
$j$ is another active coordinate, the invariance assertion of that
lemma gives
\[
U_{tj}^{(r)}(\boldsymbol X+s\boldsymbol\Sigma\boldsymbol e_k)
=
U_{tj}^{(r)}(\boldsymbol X)
\qquad\text{for every }s\in\mathbb R.
\]
But
\[
\boldsymbol\Sigma^{-1}
(\boldsymbol X+s\boldsymbol\Sigma\boldsymbol e_k)
=
\boldsymbol Y+s\boldsymbol e_k.
\]
For $s\in\mathbb R$, set
$\boldsymbol Y_s=\boldsymbol Y+s\boldsymbol e_k$ and
$\boldsymbol X_s=\boldsymbol\Sigma\boldsymbol Y_s
=\boldsymbol X+s\boldsymbol\Sigma\boldsymbol e_k$. By linearity,
\[
U_{tj}^{(r)}(\boldsymbol X_s)-U_{tj}^{(r)}(\boldsymbol X)
=
\boldsymbol v_{j,R}^{\mathsf T}(\boldsymbol Y+s\boldsymbol e_k)
-\boldsymbol v_{j,R}^{\mathsf T}\boldsymbol Y
=
s(\boldsymbol v_{j,R})_k.
\]
The residual-shift identity says the left-hand side is zero for every
$s\in\mathbb R$. Taking any $s\ne0$ therefore gives
$(\boldsymbol v_{j,R})_k=0$. Since $k$ was arbitrary in
$\mathcal B_j$, this holds for every $k\in\mathcal B_j$:
\begin{equation}\label{eq:bayes-locality-zero-coefficients}
	(\boldsymbol v_{j,R})_k=0
	\qquad\text{for every } k\in\mathcal B_j.
\end{equation}

To identify the $j$-th coefficient, relabel coordinate $j$ as coordinate $1$ and apply the
shift assertion of Lemma~\ref{LEM_SHIFT_PROPERTY}. For every
$s\in\mathbb R$,
\[
U_{tj}^{(r)}
\bigl(\boldsymbol X+s\boldsymbol\Sigma\boldsymbol e_j\bigr)
=
U_{tj}^{(r)}(\boldsymbol X)
+s\sqrt{\sigma_{j\cdot r}}.
\]
Since this perturbation replaces $\boldsymbol Y$ by
$\boldsymbol Y+s\boldsymbol e_j$, the linear representation
$U_{tj}^{(r)}(\boldsymbol X)=\boldsymbol v_{j,R}^{\mathsf T}\boldsymbol Y$
also gives
\[
U_{tj}^{(r)}
\bigl(\boldsymbol X+s\boldsymbol\Sigma\boldsymbol e_j\bigr)
=
U_{tj}^{(r)}(\boldsymbol X)
+s(\boldsymbol v_{j,R})_j.
\]
Comparing these identities for any $s\ne 0$ yields
\begin{equation}\label{eq:bayes-locality-diagonal-coefficient}
	(\boldsymbol v_{j,R})_j=\sqrt{\sigma_{j\cdot r}}>0.
\end{equation}

%

By \eqref{eq:bayes-locality-zero-coefficients}, every coordinate of
$\boldsymbol v_{j,R}$ outside $R\cup\{j\}$ is zero. Hence
\[
\operatorname{supp}(\boldsymbol v_{j,R})\subseteq R\cup\{j\}.
\]
Moreover, \eqref{eq:bayes-locality-diagonal-coefficient} shows that
$(\boldsymbol v_{j,R})_j>0$.

Now let $j,\ell\in\mathcal I_t$ be distinct. The zero-coefficient identity,
applied to each of these two active coordinates, gives
$(\boldsymbol v_{j,R})_\ell=0$ and
$(\boldsymbol v_{\ell,R})_j=0$. Their respective $j$-th and $\ell$-th coordinates are positive.
Therefore, if
$a\boldsymbol v_{j,R}+b\boldsymbol v_{\ell,R}=\boldsymbol 0$,
the $j$-th coordinate forces $a=0$, and the $\ell$-th coordinate forces
$b=0$. Thus the vectors are linearly independent, which completes the proof of Lemma~\ref{LEM_BAYES_RESIDUAL_LOCALITY}.
\end{proof}

%

\begin{flushleft}
	\textbf{Proof of Lemma~\ref{LEM_BAYES_PATH_ENCODING}}
\end{flushleft}
\begin{proof}
	Fix the finite piecewise-affine approximation. Its path set
	$\mathcal P$ is finite: a path has length at most $n$, and at each
	history there are finitely many active coordinates and finitely many
	affine pieces. We establish the slope properties first, then show
	that distinct paths define distinct affine functions for all but
	finitely many values of $\varepsilon$. Finally, we show that the
	maximizing path agrees with the rejection path of the
	piecewise-affine rule for almost every observation and all
	sufficiently small $\varepsilon$.
	
	Fix a path
	$p=((j_1,\pi_1),\ldots,(j_q,\pi_q))\in\mathcal P$, and write
	$r_t=(j_1,\ldots,j_{t-1})$ and
	$R_t=\{j_1,\ldots,j_{t-1}\}$. Its stage-$t$ affine increment is
	\[
	\Delta_{j_t,\pi_t}^{r_t}(\boldsymbol y)
	=
	s_{\pi_t}\boldsymbol v_{j_t,R_t}^{\mathsf T}\boldsymbol y
	+c_{\pi_t}-T_{r_t}.
	\]
	By Lemma~\ref{LEM_BAYES_RESIDUAL_LOCALITY}, the slope of this
	increment is supported on $R_t\cup\{j_t\}$, which is contained in
	$S(p)=\{j_1,\ldots,j_q\}$.
	
	Consider $j_t\in S(p)$. No increment before stage $t$ involves
	coordinate $j_t$, because that coordinate has not yet entered the
	rejection history. Its first contribution to
	$(\boldsymbol a_p^\varepsilon)_{j_t}$ is therefore
	\[
	\varepsilon^{t-1}s_{\pi_t}
	(\boldsymbol v_{j_t,R_t})_{j_t}.
	\]
	This coefficient is nonzero by
	Lemma~\ref{LEM_BAYES_RESIDUAL_LOCALITY} and the nonzero-slope
	property of the affine pieces. All later contributions have order at
	least $\varepsilon^t$. Hence
	$(\boldsymbol a_p^\varepsilon)_{j_t}\ne0$ for all sufficiently
	small $\varepsilon>0$. In the signed one-sided case,
	$s_{\pi_t}>0$ and
	$(\boldsymbol v_{j_t,R_t})_{j_t}>0$, so the leading coefficient is
	positive and
	$(\boldsymbol a_p^\varepsilon)_{j_t}>0$ for all sufficiently small
	$\varepsilon>0$. Coordinates outside $S(p)$ receive no contribution
	and therefore have zero coefficients. Since $\mathcal P$ is finite,
	one sufficiently small bound on $\varepsilon$ makes these support
	and sign properties hold simultaneously for every path.
	
	We next show that distinct paths define distinct affine functions,
	except at finitely many values of $\varepsilon$. Consider two paths
	that first differ at stage $t$. At that stage, one path may stop
	while the other continues; they may choose different active
	coordinates; or they may choose different affine pieces for the
	same coordinate. In the first case, the continuing path has a
	nonzero affine increment. In the second case, the increments have
	distinct linear parts, by the linear independence of the residual
	directions in Lemma~\ref{LEM_BAYES_RESIDUAL_LOCALITY}. In the
	third case, the affine pieces are distinct because duplicate pieces
	have been removed. Thus the difference between the two increments
	at their first differing stage is a nonzero affine function of
	$\boldsymbol y$.
	
	The difference between the two path functions is a polynomial in
	$\varepsilon$ whose coefficient at order $\varepsilon^{t-1}$ is
	that nonzero affine function. It is therefore not identically zero,
	and the two path functions can coincide only for finitely many
	values of $\varepsilon$. Since there are finitely many pairs of
	paths, we may exclude a finite set of values so that all path
	functions are pairwise distinct.
	
	Finally, fix an observation $\boldsymbol y$ outside the null set of
	stagewise ties and cutoff equalities. At each history, there is a
	unique maximizing active score. The step-down rule continues when
	that score exceeds the cutoff and stops otherwise; these choices
	determine a unique rejection path, denoted by $p^*$.
	
	Let $p\ne p^*$, and suppose that the paths first differ at stage
	$t$. If one path stops and the other continues, the strict stopping
	decision makes the stage-$t$ increment of $p^*$ larger. If both
	paths continue but choose different coordinates or affine pieces,
	the unique maximizing choice makes the increment of $p^*$ larger.
	Thus, for some $d_p>0$, their stage-$t$ increments differ by
	$d_p$ in favor of $p^*$. All later contributions are multiplied by
	higher powers of $\varepsilon$, so
	\[
	f_{p^*}^\varepsilon(\boldsymbol y)
	-f_p^\varepsilon(\boldsymbol y)
	=
	\varepsilon^{t-1}\{d_p+O(\varepsilon)\}>0
	\]
	for all sufficiently small $\varepsilon>0$. Because the path set
	is finite, a single sufficiently small bound works for all
	competitors $p$. Consequently, $p^*$ is the unique maximizing path
	for all sufficiently small $\varepsilon$, and its rejection set is
	the rejection set of the piecewise-affine rule. This concludes the proof of Lemma~\ref{LEM_BAYES_PATH_ENCODING}.
\end{proof}

\begin{flushleft}
	\textbf{Proof of Lemma~\ref{LEM_BAYES_FINITE_PRIOR}}
\end{flushleft}
\begin{proof}
	Write $\boldsymbol{\Omega}=\boldsymbol{\Sigma}^{-1}$, and let $\nu$
	be the law of $\boldsymbol{Y}=\boldsymbol{\Sigma}^{-1}\boldsymbol{X}$
	when $\boldsymbol{\theta}=\boldsymbol{0}$. The likelihood ratio is
\begin{equation}\label{eq:likelihood-ratio}
\frac{dP_{\boldsymbol{\theta}}^Y}{d\nu}(\boldsymbol{y})
=\exp\!\left\{
\boldsymbol{\theta}^{\mathsf T}\boldsymbol{y}
-\tfrac12\boldsymbol{\theta}^{\mathsf T}
\boldsymbol{\Omega}\boldsymbol{\theta}
\right\}.
\end{equation}
	
	For each path $p$, put
	$\boldsymbol{\theta}_p=\kappa\boldsymbol{a}_p^\varepsilon$ and
	assign it prior mass proportional to
	\[
	\exp\!\left\{
	\kappa b_p^\varepsilon+
	\frac{\kappa^2}{2}
	(\boldsymbol{a}_p^\varepsilon)^{\mathsf T}
	\boldsymbol{\Omega}\boldsymbol{a}_p^\varepsilon
	\right\}.
	\]
	If different paths give the same parameter value, their prior
	masses are added. Before this aggregation, the product of the
	unnormalized prior mass for path $p$ and its likelihood ratio is
	\[
	\exp\!\left\{
	\kappa b_p^\varepsilon+
	\frac{\kappa^2}{2}
	(\boldsymbol{a}_p^\varepsilon)^{\mathsf T}
	\boldsymbol{\Omega}\boldsymbol{a}_p^\varepsilon
	\right\}
	\exp\!\left\{
	\kappa(\boldsymbol{a}_p^\varepsilon)^{\mathsf T}\boldsymbol{y}
	-\frac{\kappa^2}{2}
	(\boldsymbol{a}_p^\varepsilon)^{\mathsf T}
	\boldsymbol{\Omega}\boldsymbol{a}_p^\varepsilon
	\right\}
	=
	\exp\{\kappa f_p^\varepsilon(\boldsymbol{y})\}.
	\]
	The normalizing constant of the prior is common to all paths and
	does not affect the posterior.
	
	For a parameter value $\boldsymbol{\theta}$ represented by one or
	more paths, let $\mathcal{P}(\boldsymbol{\theta})$ be the set of
	paths that give that value. Its unnormalized posterior weight is
	\[
	\sum_{p\in\mathcal{P}(\boldsymbol{\theta})}
	\exp\{\kappa f_p^\varepsilon(\boldsymbol{y})\}.
	\]
	Since this is a finite sum, its logarithm divided by $\kappa$
	converges, as $\kappa\to\infty$, to
	\[
	\max_{p\in\mathcal{P}(\boldsymbol{\theta})}
	f_p^\varepsilon(\boldsymbol{y}).
	\]
	By Lemma~\ref{LEM_BAYES_PATH_ENCODING}, for almost every
	$\boldsymbol{y}$ there is a unique maximizing path. Consequently,
	the posterior concentrates on the parameter value associated with
	that path. If other paths give the same parameter value, their
	contributions are included in its posterior weight.
	
	It remains to translate this concentration into a testing decision.
	By the support identity in Lemma~\ref{LEM_BAYES_PATH_ENCODING}, all
	paths giving the same parameter value have the same rejected set.
	For the two-sided problem, this set is exactly the set of coordinates
	with nonzero parameter values. For the signed one-sided problem, the
	positivity property shows that it is exactly the set of coordinates
	with positive parameter values; the other coordinates have parameter
	value zero. Thus, for almost every $\boldsymbol{y}$, the posterior
	probability that coordinate $j$ is in the alternative converges to
	one if the maximizing path rejects $j$, and to zero otherwise.
	
	Under the corresponding total misclassification loss, the Bayes rule
	rejects coordinate $j$ when its posterior alternative probability
	exceeds $1/2$. Since these probabilities converge to zero or one,
	the Bayes rule eventually agrees with the decision vector associated
	with the maximizing path, for almost every $\boldsymbol{y}$. In the
	signed one-sided case, every prior atom has nonnegative coordinates,
	so the prior is supported on $[0,\infty)^n$.
\end{proof}

\begin{flushleft}
	\textbf{Proof of Theorem~\ref{THM_BAYES_LIMITS}}
\end{flushleft}
\begin{proof}
	Write $\boldsymbol{Y}=\boldsymbol{\Sigma}^{-1}\boldsymbol{X}$ and
	$\boldsymbol{\Omega}=\boldsymbol{\Sigma}^{-1}$. Let $\nu$ be the law
	of $\boldsymbol{Y}$ at $\boldsymbol{\theta}=\boldsymbol{0}$. Since
	$\boldsymbol{\Omega}$ is positive definite, $\nu$ has a strictly
	positive density and therefore has the same null sets as Lebesgue
	measure.
	
	We first approximate the score functions. Fix a bounded observation
	set. There are finitely many histories and score functions, and all
	residuals are bounded on this set. On each resulting bounded residual
	range, choose smooth, strictly increasing score functions that
	converge to the original scores at their continuity points. For
	two-sided scores, make this approximation on the nonnegative
	residual range. A monotone function has at most countably many
	discontinuities, and each residual is a nonzero linear function of
	$\boldsymbol{Y}$. Thus the observations at which a residual is a
	discontinuity point form a null set. The sets on which two active
	scores tie or a score equals its cutoff are also null: for distinct
	active coordinates, their residual directions are linearly
	independent by Lemma~\ref{LEM_BAYES_RESIDUAL_LOCALITY}. Away from
	these exceptional sets, the finitely many strict comparisons along
	the realized rejection path persist under sufficiently accurate
	score approximation. Hence the resulting smoothed procedures
	converge almost everywhere to the original procedure on the bounded
	observation set.
	
	Fix one such smoothed procedure and let
	$B_K=\{\boldsymbol{y}:\|\boldsymbol{y}\|\leq K\}$. At each history
	$r$, choose one rate $\lambda_r>0$, common to all active coordinates,
	large enough that the transformed scores
	\[
	\varphi_j^r(u)=
	\begin{cases}
		\exp\{\lambda_r h_j^r(|u|)\},
		&\text{in the two-sided case},\\
		\exp\{\lambda_r h_j^r(u)\},
		&\text{in the signed one-sided case}
	\end{cases}
	\]
	are convex on the residual range over $B_K$. Such a choice is
	possible because the smoothed scores have derivatives bounded away
	from zero and bounded second derivatives on compact intervals. Set
	$T_r=\exp\{\lambda_r C^G_{|r|+1}\}$. The same strictly increasing
	transformation is applied to every active score and to the cutoff,
	so it preserves the ordering and the stopping decision at history
	$r$.
	
	Approximate each transformed convex score uniformly on its bounded
	range by a finite maximum of supporting affine functions. In the
	two-sided case, choose the tangency points away from zero; in the
	signed one-sided case, the tangent slopes are positive. As the
	finite grids of tangency points are refined, these affine envelopes
	converge uniformly to the transformed scores. Outside the null set
	of ties and cutoff equalities, the comparisons along the realized
	path are unchanged once the approximation is sufficiently accurate.
	Therefore, the rejection vectors of the finite piecewise-affine
	rules converge almost everywhere on $B_K$ to those of the smoothed
	procedure.
	
	Now fix one finite piecewise-affine rule. By
	Lemma~\ref{LEM_BAYES_PATH_ENCODING}, for almost every
	$\boldsymbol{y}$ its maximizing-path decision agrees with the
	rule's rejection decision for all sufficiently small admissible
	$\varepsilon>0$. Choose such an $\varepsilon$ so that the
	probability of disagreement on $B_K$ is as small as desired.
	With this $\varepsilon$ fixed, Lemma~\ref{LEM_BAYES_FINITE_PRIOR}
	shows that the finite-support Bayes rules converge almost everywhere
	to the maximizing-path decision as $\kappa\to\infty$. Choose
	$\kappa$ large enough that their probability of disagreement on
	$B_K$ is also as small as desired. In the signed one-sided case,
	Lemmas~\ref{LEM_BAYES_PATH_ENCODING} and
	\ref{LEM_BAYES_FINITE_PRIOR} ensure that every resulting prior is
	supported on $[0,\infty)^n$.
	
	We finish by choosing a diagonal sequence. Let $K_m\uparrow\infty$
	be such that $\nu(B_{K_m}^{\mathsf c})<2^{-m}$. At step $m$, choose the smoothing, affine approximation, admissible
$\varepsilon_m$, and $\kappa_m$ successively so that each of the four
successive disagreement probabilities on $B_{K_m}$ is less than
$2^{-m-2}$. By the union bound, the resulting finite-support Bayes rule
$\delta_m$ then disagrees with the original procedure on $B_{K_m}$ with
$\nu$-probability less than $2^{-m}$.
	Then
	\[
	\nu\{\delta_m\ne\delta\}
	\leq
	\nu(B_{K_m}^{\mathsf c})
	+\nu\bigl(B_{K_m}\cap\{\delta_m\ne\delta\}\bigr)
	<2^{1-m}.
	\]
	The disagreement probabilities are summable, so the
	Borel--Cantelli lemma implies that $\delta_m(\boldsymbol{y})$
	eventually equals $\delta(\boldsymbol{y})$ for $\nu$-almost every
	$\boldsymbol{y}$. Since $\nu$ and Lebesgue measure have the same
	null sets, this is the asserted Lebesgue-almost-everywhere
	convergence in the natural coordinates. This proves the theorem.
\end{proof}

\bibliography{Admissibility_References_Fixed_2026-11-19.bib}

\end{document}